\documentclass{article}

\usepackage{amsmath,amsthm,amsfonts,amssymb}

\usepackage{color}
\definecolor{myurlcolor}{rgb}{0.6,0,0}
\definecolor{mycitecolor}{rgb}{0,0,0.9}
\definecolor{myrefcolor}{rgb}{0,0,0.9}
\usepackage[pagebackref]{hyperref}
\hypersetup{colorlinks,
linkcolor=myrefcolor,
citecolor=mycitecolor,
urlcolor=myurlcolor}

\usepackage[pagebackref]{hyperref}

\usepackage{graphicx,csquotes}

\usepackage[all,2cell,cmtip]{xy}
\UseTwocells

\usepackage[normalem]{ulem}

\newcommand{\catname}[1]{\mathsf{#1}}
\newcommand{\category}[1]{\mathsf{#1}}

\let\hom\relax
\DeclareMathOperator{\hom}{hom}

\newcommand{\maps}{\colon}

\newcommand{\n}{\mathbf{n}}
\newcommand{\m}{\mathbf{m}}

\newcommand{\SC}{{\S(C)}}

\newcommand{\Ob}{\mathrm{Ob}}

\newcommand{\E}{\mathcal{E}}

\newcommand{\N}{\mathbb{N}}

\newcommand{\R}{\mathbb{R}}

\newcommand{\Boole}{\mathbb{B}}

\newcommand{\G}{\Gamma}
\newcommand{\SG}{\mathrm{SG}}
\newcommand{\DG}{\mathrm{DG}}
\newcommand{\MG}{\mathrm{MG}}
\newcommand{\MGplus}{{\mathrm{MG}^{+}}}
\newcommand{\DMG}{\mathrm{DMG}}
\newcommand{\HG}{\mathrm{HG}}

\newcommand{\Op}{\mathbf{Op}} 

\newcommand{\define}[1]{{\bf \boldmath{#1}}}

\renewcommand{\O}{\mathrm{O}}

\newcommand{\A}{\category A}
\newcommand{\B}{\category B}
\newcommand{\C}{\category C}
\newcommand{\D}{\category D}

\newcommand{\NM}{\catname{NetMod}}
\newcommand{\Mon}{\catname{Mon}}

\renewcommand{\S}{\catname{S}}
\newcommand{\Set}{\catname{Set}}
\newcommand{\Inj}{\catname{Inj}}

\newcommand{\Cat}{\catname{Cat}}
\newcommand{\CAT}{\catname{CAT}}
\newcommand{\MC}{\catname{mCat}}     
\newcommand{\BMC}{\catname{bmCat}}   
\newcommand{\SMC}{\catname{smCat}}   
\newcommand{\SSMC}{\catname{ssmCat}} 

\renewcommand{\m}{\catname{m}}     
\newcommand{\sm}{\catname{sm}}   
\newcommand{\bm}{\catname{bm}} 

\newcommand{\Fib}{\catname{Fib}}     
\newcommand{\ICat}{\catname{ICat}}   
\newcommand{\MICat}{\catname{mICat}} 
\newcommand{\BMICat}{\catname{bmICat}} 
\newcommand{\SMICat}{\catname{smICat}}  
\newcommand{\SSMICat}{\catname{ssmICat}}  

\newcommand{\twocat}{\mathcal{C}}
\newcommand{\Int}{\textstyle{\int}}

\newcommand{\inv}{^{-1}}

\DeclareMathOperator{\op}{op}

\newcommand{\To}{\Rightarrow}
\newcommand{\spl}{\mathbf{split}}

\newcommand{\Ghat}{\widehat G}
\newcommand{\Gahat}{\widehat \Gamma}
\newcommand{\gahat}{\widehat \gamma}
\newcommand{\gn}{\mathfrak g}

\newcommand{\CN}{\mathrm{O}}

\usepackage{tikz-cd}
\usetikzlibrary{arrows,positioning,fit,matrix,shapes.geometric,external}

\newcommand*\pgfdeclareanchoralias[3]{%
  \expandafter\def\csname pgf@anchor@#1@#3\expandafter\endcsname
     \expandafter{\csname pgf@anchor@#1@#2\endcsname}}

\tikzset{
    circnode/.style={
      circle, draw=red, very thin, outer sep=0.025em, minimum size=2em,
      fill=red, text centered},
    integral/.style={
      circle, draw=black, very thick, outer sep=0.025em,
      minimum size=2em, fill=blue!5, text centered},
    multiply/.style={
      circle, draw=black, very thick, outer sep=0.025em,
      minimum size=2em, fill=blue!5, text centered},
    zero/.style={
      circle, draw=black, very thick, minimum size=0.15cm, fill=black,
      inner sep=0, outer sep=0},
    bang/.style={
      circle, draw=black, very thick, minimum size=0.15cm, fill=green!10,
      inner sep=0, outer sep=0},
    delta/.style={
      regular polygon, regular polygon sides=3, minimum size=0.4cm, inner
      sep=0, outer sep=0.025em, draw=black, very thick, fill=green!10},
    codelta/.style={
      regular polygon, regular polygon sides=3, shape border rotate=180, minimum size=0.4cm,
      inner sep=0, outer sep=0.025em, draw=black, very thick, fill=green!10},
    plus/.style={
      regular polygon, regular polygon sides=3, shape border rotate=180, minimum size=0.4cm,
      inner sep = 0, outer sep=0.025em, draw=black, very thick, fill=black},
    coplus/.style={
      regular polygon, regular polygon sides=3, minimum size=0.4cm,
      inner sep = 0, outer sep=0.025em, draw=black, very thick, fill=black},
    sqnode/.style={
      regular polygon, regular polygon sides=4, minimum size=2.6em,
      draw=black, very thick, inner sep=0.2em, outer sep=0.025em,
      fill=yellow!10, text centered},
    bigcirc/.style={
      circle, draw=black, very thick, text width=1.6em, outer sep=0.025em,
      minimum height=1.6em, fill=blue!5, text centered}
}

\tikzstyle{tri}=[regular polygon,regular polygon sides=3,shape border rotate=1
80,fill=none,draw=black]

\pgfdeclareanchoralias{regular polygon}{corner 2}{left copy}
\pgfdeclareanchoralias{regular polygon}{corner 3}{right copy}
\pgfdeclareanchoralias{regular polygon}{corner 3}{addend}
\pgfdeclareanchoralias{regular polygon}{corner 2}{summand}

\usepackage{tikz}
\usetikzlibrary{positioning}
\definecolor {processblue}{cmyk}{0.9,0.5,0,0}

\tikzstyle{simple}=[-,line width=2.000]
\tikzstyle{arrow}=[-,postaction={decorate},decoration={markings,mark=at position .5 with {\arrow{>}}},line width=1.100]
\pgfdeclarelayer{edgelayer}
\pgfdeclarelayer{nodelayer}
\pgfdeclarelayer{bg}
\pgfsetlayers{bg,edgelayer,nodelayer,main}

\tikzstyle{none}=[inner sep=-1pt]
\tikzstyle{species}=[circle,fill=none,draw=black,scale=0.75]
\tikzstyle{transition}=[rectangle,fill=none,draw=black,scale=1.15]
\tikzstyle{empty}=[circle,fill=none, draw=none]
\tikzstyle{inputdot}=[circle,fill=black,draw=black, scale=.5]
\tikzstyle{dot}=[circle,fill=black,draw=black]
\tikzstyle{bounding}=[circle,dashed, fill=none,draw=black, scale=9.00]
\tikzstyle{triplebounding}=[circle,dashed, fill=none,draw=black, scale=30.00]
\tikzstyle{simple}=[-,draw=black,line width=1.000]
\tikzstyle{inarrow}=[-,draw=black,postaction={decorate},decoration={markings,mark=at position .5 with {\arrow{>}}},line width=1.000]
\tikzstyle{tick}=[-,draw=black,postaction={decorate},decoration={markings,mark=at position .5 with {\draw (0,-0.1) -- (0,0.1);}},line width=1.000]
\tikzstyle{inputarrow}=[->,draw=black, shorten >=.05cm]

\newcommand{\beq}{\begin{equation}}
\newcommand{\eeq}{\end{equation}}
\newcommand{\beqa}{\begin{eqnarray}}
\newcommand{\eeqa}{\end{eqnarray}}

\theoremstyle{plain}

\newtheorem{thm}{Theorem}

\newtheorem{lem}[thm]{Lemma}
\newtheorem{prop}[thm]{Proposition}

\theoremstyle{remark}

\theoremstyle{definition}
\newtheorem{defn}[thm]{Definition}

\newtheorem{ex}[thm]{Example}

\begin{document}

\title{Network Models}
\author{
\begin{tabular}{cccc}
\small John C.\ Baez\footnote{Department of Mathematics, University of California, Riverside CA, 92521, USA}{ }\footnote{Centre for Quantum Technologies, National University of Singapore, 117543, Singapore} & \small John Foley\footnote{Metron, Inc., 1818 Library St., Suite 600, Reston, VA 20190, USA} & \small Joe Moeller$^*$ & \small Blake S.\ Pollard\footnote{NIST, Mail Stop 8970, 100 Bureau Drive, Gaithersburg, MD 20899, USA}
\\
\small baez@math.ucr.edu & \small foley@metsci.com & \small moeller@math.ucr.edu & \small 
blake.pollard@nist.gov
\end{tabular}
}

\date{\today}

\maketitle 

\begin{abstract}
    \noindent Networks can be combined in various ways, such as overlaying one on top of another or setting two side by side. We introduce `network models' to encode these ways of combining networks. Different network models describe different kinds of networks. We show that each network model gives rise to an operad, whose operations are ways of assembling a network of the given kind from smaller parts. Such operads, and their algebras, can serve as tools for designing networks. Technically, a network model is a lax symmetric monoidal functor from the free symmetric monoidal category on some set to $\Cat$, and the construction of the corresponding operad proceeds via a symmetric monoidal version of the Grothendieck construction.
\end{abstract}

\section{Introduction}
\label{sec:intro}

In this paper we study operads suited for designing networks. These could be networks where the vertices represent fixed or moving agents and the edges represent communication channels. More generally, they could be networks where the vertices represent entities of various types, and the edges represent relationships between these entities, e.g.\ that one agent is committed to take some action involving the other. This paper arose from an example where the vertices represent planes, boats and drones involved in a search and rescue mission in the Caribbean \cite{CommNet, CompTask}. However, even for this one example, we want a flexible formalism that can handle networks of many kinds, described at a level of detail that the user is free to adjust.

To achieve this flexibility, we introduce a general concept of `network model'. Simply put, a network model is a \emph{kind} of network. Any network model gives an operad whose operations are ways to build larger networks of this kind by gluing smaller ones. This operad has a `canonical' algebra where the operations act to assemble networks of the given kind. But it also has other algebras, where it acts to assemble networks of this kind \emph{equipped with extra structure and properties}. This flexibility is important in applications. 

What exactly is a `kind of network'?  At the crudest level, we can model networks as simple graphs. If the vertices are agents of some sort and the edges represent communication channels, this means we allow at most one channel between any pair of agents. However, simple graphs are too restrictive for many applications. If we allow multiple communication channels between a pair of agents, we should replace simple graphs with `multigraphs'. Alternatively, we may wish to allow directed channels, where the sender and receiver have different capabilities: for example, signals may only be able to flow in one direction. This requires replacing simple graphs with `directed graphs'. To combine these features we could use `directed multigraphs'. It is also important to consider graphs with colored vertices, to specify different types of agents, and colored edges, to specify different types of channels. This leads us to `colored directed multigraphs'. All these are examples of what we mean by a `kind of network'. Even more complicated kinds, such as hypergraphs or Petri nets, are likely to become important as we proceed. 
Thus, instead of separately studying all these kinds of networks, we introduce a unified notion that subsumes all these variants: a `network model'. Namely, given a set $C$ of `vertex colors', a \define{network model} $F$ is a lax symmetric monoidal functor $F \maps \S(C) \to \Cat$, where $\S(C)$ is the free strict symmetric monoidal category on $C$ and $\Cat$ is the category of small categories, considered with its cartesian monoidal structure. Unpacking this definition takes a little work. It simplifies in the special case where $F$ takes values in $\Mon$, the category of monoids. It simplifies further when $C$ is a singleton, since then $\S(C)$ is the groupoid $\S$, where objects are natural numbers and morphisms from $m$ to $n$ are bijections $\sigma \maps \{1,\dots,m\} \to \{1,\dots,n\}$.
If we impose both these simplifying assumptions, we have what we call a \define{one-colored network model}: a lax symmetric monoidal functor $F \maps \S \to \Mon$. 
As we shall see, the network model of simple graphs is a one-colored network model, and so are many other motivating examples.

Joyal began an extensive study of functors $F \maps \S \to \Set$, which are now commonly called `species' \cite{BLL,Joyal,Joyal2}.
Any type of extra structure that can be placed on finite sets and transported along bijections defines a species if we take $F(n)$ to be the set of structures that can be placed on the set $\{1, \dots, n\}$. From this perspective, a one-colored network model is a species with some extra operations. 

This perspective is helpful for understanding what a
one-colored network model $F \maps \S \to \Mon$ is actually like. If we call elements of $F(n)$ `networks with $n$ vertices', then:
\begin{enumerate}
    \item Since $F(n)$ is a monoid, we can \define{overlay} two networks with the same number of vertices and get a new one. We denote this operation by
    \[
        \cup \colon F(n) \times F(n) \to F(n) . 
    \]
    For example:
    \[\scalebox{0.8}
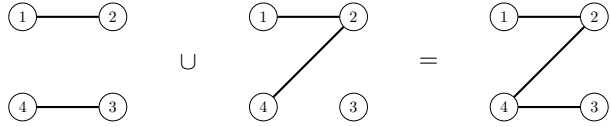
{
    \begin{tikzpicture}
    	\begin{pgfonlayer}{nodelayer}
    	    \node [style=none, scale = 1.2] () at (5,2) {$\cup$};
    	    \node [style=none, scale = 1.2] () at (9,2) {=};
    		\node [style=species]  (1) at (3.75, 2.75) {2};
    		\node [style=species]  (2) at (2.25, 2.75) {1};
    		\node [style=species]  (3) at (2.25, 1.25) {4};
    		\node [style=species]  (4) at (3.75, 1.25) {3};
    		\node [style=species]  (5) at (7.75, 2.75) {2};
    		\node [style=species]  (6) at (6.25, 2.75) {1};
    		\node [style=species]  (7) at (6.25, 1.25) {4};
    		\node [style=species]  (8) at (7.75, 1.25) {3};
    		\node [style=species]  (9) at (11.75, 2.75) {2};
    		\node [style=species]  (10) at (10.25, 2.75) {1};
    		\node [style=species]  (11) at (10.25, 1.25) {4};
    		\node [style=species]  (12) at (11.75, 1.25) {3};
    	\end{pgfonlayer}
    	\begin{pgfonlayer}{edgelayer}
    		\draw [style=simple] (2) to (1);
    		\draw [style=simple] (3) to (4);
    		\draw [style=simple] (6) to (5);
    		\draw [style=simple] (5) to (7);
    		%
    		\draw [style=simple] (10) to (9);
    		\draw [style=simple] (9) to (11);
    		\draw [style=simple] (11) to (12);
    	\end{pgfonlayer}
    \end{tikzpicture}}\]
     \item Since $F$ is a functor, the group $S_n$ acts on the monoid $F(n)$. Thus, for each $\sigma \in S_n$, we have a monoid automorphism that we call
    \[\sigma \maps F(n) \to F(n)  . \]
    For example, if $\sigma = (2\,3) \in S_3$, then
    \[\scalebox{0.8}
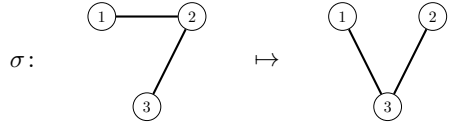
{
    \begin{tikzpicture}
    	\begin{pgfonlayer}{nodelayer}
    	    \node [style=none, scale = 1.2] () at (5,2) {$\sigma\maps$};
    	    \node [style=none, scale = 1.2] () at (9,2) {$\mapsto$};
    		\node [style=species]  (1) at (7.75, 2.75) {2};
    		\node [style=species]  (2) at (6.25, 2.75) {1};
    		\node [style=species]  (3) at (7, 1.25) {3};
    		\node [style=species]  (13) at (11.75, 2.75) {2};
    		\node [style=species]  (14) at (10.25, 2.75) {1};
    		\node [style=species]  (15) at (11, 1.25) {3};
    	\end{pgfonlayer}
    	\begin{pgfonlayer}{edgelayer}
    		\draw [style=simple] (2) to (1);
    		\draw [style=simple] (1) to (3);
    		\draw [style=simple] (14) to (15);
    		\draw [style=simple] (15) to (13);
    	\end{pgfonlayer}
    \end{tikzpicture}}\]
    \item Since $F$ is lax monoidal, we have an operation
    \[
        \sqcup \colon F(m) \times F(n) \to F(m+n)
    \]
    We call this operation the \define{disjoint union} of networks. For example:
    \[\scalebox{0.8}
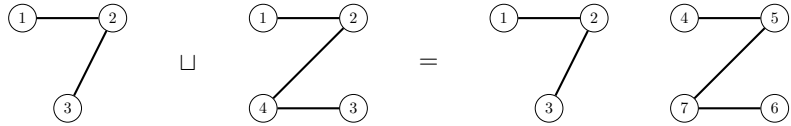
{
    \begin{tikzpicture}
    	\begin{pgfonlayer}{nodelayer}
    	    \node [style=none, scale = 1.2] () at (5,2) {$\sqcup$};
    	    \node [style=none, scale = 1.2] () at (9,2) {=};
    		\node [style=species]  (1) at (3.75, 2.75) {2};
    		\node [style=species]  (2) at (2.25, 2.75) {1};
    		\node [style=species]  (3) at (3, 1.25) {3};
    		\node [style=species]  (5) at (7.75, 2.75) {2};
    		\node [style=species]  (6) at (6.25, 2.75) {1};
    		\node [style=species]  (7) at (6.25, 1.25) {4};
    		\node [style=species]  (8) at (7.75, 1.25) {3};
    		\node [style=species]  (9) at (14.75, 2.75) {5};
    		\node [style=species]  (10) at (13.25, 2.75) {4};
    		\node [style=species]  (11) at (13.25, 1.25) {7};
    		\node [style=species]  (12) at (14.75, 1.25) {6};
    		\node [style=species]  (13) at (11.75, 2.75) {2};
    		\node [style=species]  (14) at (10.25, 2.75) {1};
    		\node [style=species]  (15) at (11, 1.25) {3};
    	\end{pgfonlayer}
    	\begin{pgfonlayer}{edgelayer}
    		\draw [style=simple] (2) to (1);
    		\draw [style=simple] (1) to (3);
    		\draw [style=simple] (6) to (5);
    		\draw [style=simple] (5) to (7);
    		\draw [style=simple] (7) to (8);
    		\draw [style=simple] (10) to (9);
    		\draw [style=simple] (9) to (11);
    		\draw [style=simple] (11) to (12);
    		\draw [style=simple] (14) to (13);
    		\draw [style=simple] (15) to (13);
    	\end{pgfonlayer}
    \end{tikzpicture}}\]
\end{enumerate}
The first two operations are present whenever we have a 
functor $F \maps \S \to \Mon$. The last two are present whenever we have a lax symmetric monoidal functor $F \maps \S \to \Set$. When $F$ is a one-colored network model we have all three---and unpacking the definition further, we see that they obey some equations, which we list in Theorem \ref{thm:equations}. For example, the `interchange law' 
\[(g \cup g') \sqcup (h \cup h') = (g \sqcup h) \cup (g' \sqcup h') \]
holds whenever $g,g' \in F(m)$ and $h, h' \in F(n)$.

In Section \ref{sec:netmod} we study one-colored network models more formally, and give many examples. In Section \ref{sec:models_from_monoids} we describe a systematic procedure for getting one-colored network models from monoids. In Section \ref{sec:netmod_C} we study general network models and give examples of these. In Section \ref{sec:cat_netmod} we describe a category $\NM$ of network models, and show that the procedure for getting network models from monoids is functorial. We also make $\NM$ into a symmetric monoidal category, and give examples of how to build new networks models by tensoring old ones.

Our main result is that any network model gives a typed operad, also known as a `colored operad' or `symmetric multicategory' \cite{Yau}. A typed operad describes ways of sticking together things of various types to get new things of various types. An algebra of the operad gives a particular specification of these things and the results of sticking them together. A bit more precisely, a typed operad $O$ has:
\begin{itemize}
    \item a set $T$ of \define{types},
    \item sets of \define{operations} $O(t_1,...,t_n ; t)$ where $t_i, t \in T$,
    \item ways to \define{compose} any operation
    \[f \in O(t_1,\dots,t_n ;t) \]
    with operations 
    \[g_i \in O(t_{i1},\dots,t_{i k_i}; t_i)   \qquad (1 \le i \le n) \] 
    to obtain an operation 
    \[f \circ (g_1,\dots,g_n) \in O(t_{1i}, \dots, t_{1k_1}, \dots, 
    t_{n1}, \dots t_{n k_n}; t), \]
    \item and ways to permute the arguments of operations,
    \end{itemize}
    which obey some rules \cite{Yau}. An algebra $A$ of $O$ specifies a set $A(t)$ for each type $t \in T$ such that the operations of $O$ act on these sets. Thus, it has:
    \begin{itemize}
    \item for each type $t \in T$, a set $A(t)$ of \define{things} of type $t$,
    \item ways to \define{apply} any operation
    \[f \in O(t_1, \dots, t_n ; t) \]
    to things
    \[a_i \in A(t_i)  \qquad (1 \le i \le n) \]
    to obtain a thing
    \[\alpha(f)(a_1, \dots, a_n) \in A(t). \]
\end{itemize}
Again, we demand that some rules hold \cite{Yau}. 

 In Thm.\ \ref{thm:one-colored_network_operads} we describe the typed operad $\CN_F$ arising from a one-colored network model $F$. The set of types is $\N$, since we can think of `network with $n$ vertices' as a type. The sets of operations are given as follows:
\[\CN_F(n_1, \dots, n_k; n) = \left\{ 
\begin{array}{cl}  S_n \times F(n) & \textrm{if } 
n_1 + \cdots + n_k = n \\
\emptyset & \textrm{otherwise.} 
\end{array}  \right. \]
The key idea here is that we can overlay a network in $F(n)$ on the disjoint union of networks with $n_1, \dots, n_k$ vertices and get a new network with $n$ vertices as long as $n_1 + \cdots n_k = n$. We can also permute the vertices; this accounts for the group $S_n$. But the most important fact is that \emph{networks serve as operations to assemble networks}, thanks to our ability to overlay them.

Using this fact, we show in Ex.\ \ref{ex:canonical_algebra} that the operad $\CN_F$ has a canonical algebra $A_F$ whose elements are simply networks of the kind described by $F$:
\[A_F(n) = F(n) .\]
In this algebra any operation
\[(\sigma,g) \in  \CN_F(n_1, \dots , n_k; n) = 
S_n \times F(n) \] 
acts on a $k$-tuple of networks
\[h_i \in A_F(n_i) = F(n_i)   \qquad (1 \le i \le k) \]
to give the network
\[\alpha(\sigma,g)(h_1, \dots, h_k) =  g \cup \sigma(h_1 \sqcup \cdots \sqcup h_k) \in A_F(n) .\]
In other words, we first take the disjoint union of the networks $h_i$, then permute their vertices with $\sigma$, and then overlay the network $g$.

An example is in order, since the generality of the formalism may hide the simplicity of the idea. The easiest example of our theory is the network model for simple graphs. In Ex.\ \ref{ex:simple_graph} we describe a one-colored network model $\SG \maps \S \to \Mon$ such that $\SG(n)$ is the collection of simple graphs with vertex set $\n = \{1,\dots,n\}$. Such a simple graph is really a collection of 2-element subsets of $\n$, called `edges'. Thus, we may overlay simple graphs $g,g' \in \SG(n)$ by taking their union $g \cup g'$. This operation makes $\SG(n)$ into a monoid. 

Now consider an operation $f \in \CN_\SG(3,4,2;9)$. This is an element of $S_9 \times \SG(9)$: a permutation of the set $\{1,\dots, 9\}$ together with a simple graph having this set of vertices. If we take the permutation to be the identity for simplicity, this operation is just a simple graph $g \in \SG(9)$. We can draw an example as follows:
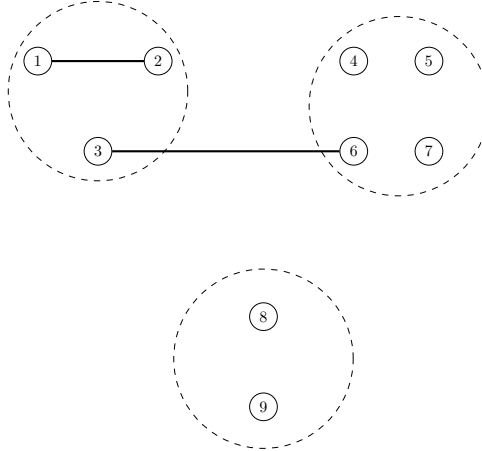
\[\scalebox{0.8}{
\begin{tikzpicture}
	\begin{pgfonlayer}{nodelayer}
		\node [style=species] (0) at (2, 1) {$3$};
		\node [style=species] (1) at (4.75, -3.25) {$9$};
		\node [style=species] (2) at (7.5, 2.5) {$5$};
		\node [style=species] (3) at (1, 2.5) {$1$};
		\node [style=none] (4) at (0, 3) {};
		\node [style=none] (5) at (7, -2.5) {};
		\node [style=species] (6) at (4.75, -1.75) {$8$};
		\node [style=bounding] (7) at (4.75, -2.45) {};
		\node [style=bounding] (8) at (2, 2) {};
		\node [style=none] (9) at (8.75, 3) {};
		\node [style=species] (10) at (6.25, 2.5) {$4$};
		\node [style=species] (11) at (6.25, 1) {$6$};
		\node [style=bounding] (12) at (7, 1.75) {};
		\node [style=species] (13) at (3, 2.5) {$2$};
		\node [style=species] (14) at (7.5, 1) {$7$};
	\end{pgfonlayer}
	\begin{pgfonlayer}{edgelayer}
		\draw [style=simple] (0) to (11);
		\draw [style=simple] (3) to (13);
	\end{pgfonlayer}
\end{tikzpicture}
}\]
The dashed circles indicate that we are thinking of this simple graph as an element of $\CN(3,4,2;9)$: an operation that can be used to assemble simple graphs with 3, 4, and 2 vertices, respectively, to produce one with 9 vertices.

Next let us see how this operation acts on the canonical algebra $A_\SG$, whose elements are simple graphs. Suppose we have elements $a_1 \in A_\SG(3)$, $a_2 \in A_\SG(4)$ and $a_3 \in A_\SG(2)$:
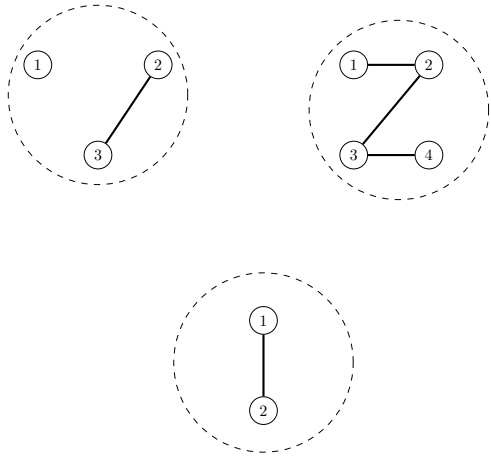
\[
\scalebox{0.8}{
\begin{tikzpicture}
	\begin{pgfonlayer}{nodelayer}
		\node [style=species] (0) at (2, 1) {$3$};
		\node [style=species] (1) at (4.75, -3.25) {$2$};
		\node [style=species] (2) at (7.5, 2.5) {$2$};
		\node [style=species] (3) at (1, 2.5) {$1$};
		\node [style=none] (4) at (0, 3) {};
		\node [style=none] (5) at (7, -2.5) {};
		\node [style=species] (6) at (4.75, -1.75) {$1$};
		\node [style=bounding] (7) at (4.75, -2.45) {};
		\node [style=bounding] (8) at (2, 2) {};
		\node [style=none] (9) at (8.75, 3) {};
		\node [style=species] (10) at (6.25, 2.5) {$1$};
		\node [style=species] (11) at (6.25, 1) {$3$};
		\node [style=bounding] (12) at (7, 1.75) {};
		\node [style=species] (13) at (3, 2.5) {$2$};
		\node [style=species] (14) at (7.5, 1) {$4$};
	\end{pgfonlayer}
	\begin{pgfonlayer}{edgelayer}
		\draw [style=simple] (13) to (0);
		\draw [style=simple] (10) to (2);
		\draw [style=simple] (2) to (11);
		\draw [style=simple] (11) to (14);
		\draw [style=simple] (6) to (1);
	\end{pgfonlayer}
\end{tikzpicture}
}
\]
We can act on these by the operation $f$ to obtain $\alpha(f)(a_1,a_2,a_3) \in A_\SG(9)$. It looks like this:
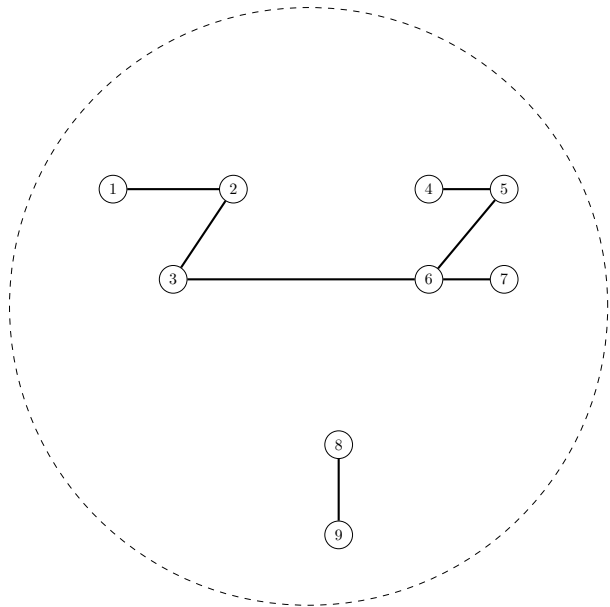
\[\scalebox{0.8}{
\begin{tikzpicture}
	\begin{pgfonlayer}{nodelayer}
		\node [style=species] (0) at (2, 1) {$3$};
		\node [style=species] (1) at (4.75, -3.25) {$9$};
		\node [style=species] (2) at (7.5, 2.5) {$5$};
		\node [style=species] (3) at (1, 2.5) {$1$};
		\node [style=none] (4) at (0, 3) {};
		\node [style=none] (5) at (7, -2.5) {};
		\node [style=species] (6) at (4.75, -1.75) {$8$};
		\node [style=none] (9) at (8.75, 3) {};
		\node [style=species] (10) at (6.25, 2.5) {$4$};
		\node [style=species] (11) at (6.25, 1) {$6$};
		\node [style=species] (13) at (3, 2.5) {$2$};
		\node [style=species] (14) at (7.5, 1) {$7$};
		\node [style=triplebounding] (15) at (4.25, .55) {};
	\end{pgfonlayer}
	\begin{pgfonlayer}{edgelayer}
		\draw [style=simple] (3) to (13);
		\draw [style=simple] (13) to (0);
		\draw [style=simple] (10) to (2);
		\draw [style=simple] (0) to (11);
		\draw [style=simple] (2) to (11);
		\draw [style=simple] (11) to (14);
		\draw [style=simple] (6) to (1);
	\end{pgfonlayer}
\end{tikzpicture}}\]
We have simply taken the disjoint union of $a_1$, $a_2$, and $a_3$ and then overlaid $g$, obtaining a simple graph with 9 vertices.

The canonical algebra is one of the simplest algebras of the operad $O_\SG$. We can define many more interesting algebras for this operad. For example, we might wish to use this operad to describe communication networks where the communicating entities have locations and the communication channels have limits on their range. To include location data, we can choose $A(n)$ for $n \in \N$ to be the set of all graphs with $n$ vertices where each vertex is an actual point in the plane $\R^2$. To handle range-limited communications, we could instead choose $A(n)$ to be the set of all graphs with $n$ vertices in $\R^2$ where an edge is permitted between two vertices only if their Euclidean distance is less than some specified value. This still gives a well-defined algebra: when we apply an operation, we simply omit those edges from the resulting graph that would violate this restriction.

Besides the plethora of interesting algebras for the operad $O_\SG$, and useful homomorphisms between these, one can also modify the operad by choosing another network model. This provides additional flexibility in the formalism. Different network models give different operads, and the construction of operads from network models is functorial, so morphisms of network models give morphisms of operads. 

The technical heart of our paper is Section \ref{sec:Grothendieck}, which provides the machinery to construct operads from network models in a functorial way. This section is of independent interest, because it describes enhancements of the well-known `Grothendieck construction' of the category of elements $\Int F$ of a functor $F \maps C \to \Cat$, where $C$ is any small category. For example, suppose $C$ is symmetric monoidal and $F \maps C \to \Cat$ is lax symmetric monoidal, where we give $\Cat$ its cartesian symmetric monoidal structure. Then we show $\Int F$ is symmetric monoidal. Moreover, we show that the construction sending the lax symmetric monoidal functor $F$ to the symmetric monoidal category $\Int F$ is functorial. 

In Section \ref{sec:operads} we apply this machinery to build operads from network models. In Section \ref{sec:algebras} we describe some algebras of these operads, and in Ex.\ \ref{ex:range_limit_algebra} we discuss an algebra whose elements are networks of range-limited communication channels. In future work we plan to give many more detailed examples, and to explain how these algebras, and the homomorphisms between them, can be used to design and optimize networks.

\section{One-colored network models}
\label{sec:netmod}

We begin with a special class of network models: those where the vertices of the network have just one color. To define these, we use $\S$ to stand for a skeleton of the groupoid of finite sets and bijections:

\begin{defn}
    Let $\S$, the \define{symmetric groupoid}, be the groupoid for which:
    \begin{itemize}
        \item objects are natural numbers $n \in \N$,
        \item a morphism from $m$ to $n$ is a bijection $\sigma \colon \{1,\dots,m\} \to
        \{1,\dots,n\}$
    \end{itemize}
    and bijections are composed in the usual way.
\end{defn}

\noindent
There are no morphisms in $\S$ from $m$ to $n$ unless $m = n$. For each $n \in \N$, the morphisms $\sigma \maps n \to n$ form the symmetric group $S_n$. It is convenient to write $\n$ for the set $\{1,\dots,n\}$, so that a morphism $\sigma \maps n \to n$ in $\S$ is the same as a bijection $\sigma \maps \n \to \n$.

There is a functor $+ \maps \S \times \S \to \S$ defined as follows. Given $m, n \in \N$ we let $m + n$ be the usual sum, and given $\sigma \in S_m$ and $\tau \in S_n$, let $\sigma+\tau \in S_{m+n}$ be as follows:
\begin{equation}
\label{eq:plus}
    (\sigma + \tau)(j)=
    \begin{cases}
        \sigma(j)&\text{if } j\leq m
        \\\tau(j-m)+m&\text{otherwise.}
    \end{cases}
\end{equation}
For objects $m, n \in \S$, let $B_{m,n}$ be the block permutation of $m+n$ which swaps the first $m$ with the last $n$. For example $B_{4,3} \maps 7 \to 7$ is the permutation $(1473625)$:
\[\begin{tikzpicture}
	\begin{pgfonlayer}{nodelayer}
		\node [style=empty] (1) at (1, 1.5) {};
		\node [style=empty] (2) at (1.5, 1.5) {};
		\node [style=empty] (3) at (2, 1.5) {};
		\node [style=empty] (4) at (2.5, 1.5) {};
		\node [style=empty] (5) at (3, 1.5) {};
		\node [style=empty] (6) at (3.5, 1.5) {};
		\node [style=empty] (7) at (4, 1.5) {};
		\node [style=empty] (1a) at (1, 0) {};
		\node [style=empty] (2a) at (1.5, 0) {};
		\node [style=empty] (3a) at (2, 0) {};
		\node [style=empty] (4a) at (2.5, 0) {};
		\node [style=empty] (5a) at (3, 0) {};
		\node [style=empty] (6a) at (3.5, 0) {};
		\node [style=empty] (7a) at (4, 0) {};
	\end{pgfonlayer}
	\begin{pgfonlayer}{edgelayer}
		\draw [style=simple] (1.center) to (4a.center);
		\draw [style=simple] (2.center) to (5a.center);
		\draw [style=simple] (3.center) to (6a.center);
		\draw [style=simple] (4.center) to (7a.center);
		\draw [style=simple] (5.center) to (1a.center);
		\draw [style=simple] (6.center) to (2a.center);
		\draw [style=simple] (7.center) to (3a.center);
	\end{pgfonlayer}
\end{tikzpicture} \]
The tensor product $+$ and braiding $B$ give $\S$ the structure of a strict symmetric monoidal category. This follows as a special case of Prop.\ \ref{prop:free}. 

\begin{defn}
\label{defn:network_model}
    A \define{one-colored network model} is a lax symmetric monoidal functor 
    \[F \maps  \S \to \Mon .\]
    Here $\Mon$ is the category with monoids as objects and monoid homomorphisms as morphisms, considered with its cartesian monoidal structure.
\end{defn}
\noindent For lax symmetric monoidal functors, see Mac Lane \cite{Mac Lane}, who however calls them just symmetric monoidal functors. 

Many examples of network models are given below.  A pedestrian way to verify that these examples really are network models is to use the following result:

\begin{thm}
\label{thm:equations}
    A one-colored network model $F \maps \S \to \Mon$ is the same as:
    \begin{itemize}
        \item a family of sets $\{F(n)\}_{n\in \N}$
        \item distinguished \define{identity} elements $e_n \in F(n)$
        \item a family of \define{overlaying} functions $\cup \maps F(n) \times F(n) \to F(n)$
        \item a bijection $\sigma \maps F(n) \to F(n)$ for each $\sigma \in S_n$
        \item a family of \define{disjoint union} functions $\sqcup \maps F(m) \times F(n) \to F(m+n)$
    \end{itemize}

    satisfying the following equations:
    \begin{enumerate}
        \item $e_n \cup g = g \cup e_n = g$
        \item $g_1 \cup (g_2 \cup g_3) = (g_1 \cup g_2) \cup g_3$
        \item $\sigma(g_1 \cup g_2) = \sigma g_1 \cup \sigma g_2$
        \item $\sigma e_n = e_n$
        \item $(\sigma_2 \sigma_1) g = \sigma_2 (\sigma_1 g)$
        \item $1 (g) = g$
        \item $(g_1 \cup g_2) \sqcup (h_1 \cup h_2) = (g_1 \sqcup h_1) \cup (g_2 \sqcup h_2)$
        \item $e_m \sqcup e_n = e_{m+n}$
        \item $\sigma g \sqcup \tau h = (\sigma + \tau) (g \sqcup h)$
        \item $g_1 \sqcup (g_2 \sqcup g_3) = (g_1 \sqcup g_2) \sqcup g_3$
        \item $e_0 \sqcup g = g \sqcup e_0 = g$
        \item $B_{m,n} (h \sqcup g) = g \sqcup h$
    \end{enumerate}
    for $g, g_i \in F(n)$, $h, h_i \in F(m)$, $\sigma, \sigma_i \in S_n$, $\tau \in S_m$, and $1$ the identity of $S_n$.
\end{thm}
\begin{proof}
    Having a functor $F \maps \S \to \Mon$ is equivalent to having the first four items satisfying equations 1--6. The binary operation $\cup$ gives the set $F(n)$ the structure of a monoid, with $e_n$ acting as the identity. Equation 1 tells us $e_n$ acts as an identity, and Equation 2 gives the associativity of $\cup$. Equations 3 and 4 tell us that $\sigma$ is a monoid homomorphism. Equations 5 and 6 say that the map $(\sigma,g) \mapsto \sigma g$ defines an action of $S_n$ on $F(n)$ for each $n$. All of these actions together give us the functor $F \maps \S \to \Mon$.
    
    That the functor is lax monoidal is equivalent to having item 5 satisfying Equations 7--11. Equations 7 and 8 tell us that $\sqcup$ is a family of monoid homomorphisms. Equation 9 tells us that it is a natural transformation. Equation 10 tells us that the associativity hexagon diagram for lax monoidal functors commutes for $F$. Equation 11 implies the commutativity of the left and right unitor square diagrams. That the lax monoidal functor is symmetric is equivalent to Equation 12. It tells us that the square diagram for symmetric monoidal functors commutes for $F$. 
\end{proof}

This is one of the simplest examples of a network model:

\begin{ex}[\textbf{Simple graphs}] \label{ex:simple_graph}
    Let a \define{simple graph} on a set $V$ be a set of 2-element subsets of $V$, called \define{edges}. There is a one-colored network model $\SG \maps \S \to \Mon$ such that $\SG(n)$ is the set of simple graphs on $\n$.
    
    To construct this network model, we make $\SG(n)$ into a monoid where the product of simple graphs $g_1, g_2 \in \SG(n)$ is their union $g_1 \cup g_2$. Intuitively speaking, to form their union, we `overlay' these graphs by taking the union of their sets of edges. The simple graph on $\n$ with no edges acts as the unit for this operation. The groups $S_n$ acts on the monoids $\SG(n)$ by permuting vertices, and these actions define a functor $\SG \maps \S \to \Mon$.
    
    Given simple graphs $g \in \SG(m)$ and $h \in \SG(n)$ we define $g \sqcup h \in \SG(m + n)$ to be their disjoint union. This gives a monoid homomorphism $\sqcup \maps \SG(m) \times \SG(n) \to \SG(m + n)$ because 
    \[(g_1 \cup g_2) \sqcup (h_1 \cup h_2) = (g_1 \sqcup h_1) \cup (g_2 \sqcup h_2). \]
    This in turn gives a natural transformation with components
    \[\sqcup_{m, n} \maps \SG(m) \times \SG(n) \to \SG(m + n), \] 
    which makes $\SG$ into lax symmetric monoidal functor. 
    
    One can prove this construction really gives a network model using either Thm.\ \ref{thm:equations}, which requires verifying a list of equations, or Thm.\ \ref{thm:graph_model}, which gives a  general procedure for getting a network model from a monoid $M$ by letting elements of $\G_M(n)$ be maps from the complete graph on $\n$ to $M$. If we take $M  = \Boole = \{F,T\}$ with `or' as the monoid operation, this procedure gives the network model $\SG = \G_\Boole$. We explain this in Ex.\ \ref{ex:simple_graph_2}.
\end{ex}
 
There are many other kinds of graph, and many of them give network models:

\begin{ex}[\textbf{Directed graphs}] \label{ex:directed_graph}
    Let a \define{directed graph} on a set $V$ be a collection of ordered pairs $(i,j) \in V^2$ such that $i \ne j$. These pairs are called \define{directed edges}. There is a network model $\DG \maps \S \to \Mon$ such that $\DG(n)$ is the set of directed graphs on $\n$. As in Ex.\
     \ref{ex:simple_graph}, the monoid operation on $\DG(n)$ is union.
\end{ex}

\begin{ex}[\textbf{Multigraphs}] \label{ex:multigraph}
    Let a \define{multigraph} on a set $V$ be a multiset of 2-element subsets of $V$. If we define $\MG(n)$ to be the set of multigraphs on $\n$, then there are at least two natural choices for the monoid operation on $\MG(n)$. 
    The most direct generalization of $\SG$ of Ex.\
     \ref{ex:simple_graph} is the network model $\MG \maps \S \to \Mon$ with 
     values $(\MG(n), \cup)$ where $\cup$  is now union of edge multisets.
     That is, the multiplicity of $\{ i, j \}$ in  $g \cup h$ is maximum of the  multiplicity of $\{ i, j \}$ in $g$ and the  multiplicity of $\{ i, j \}$ in $h$.
    Alternatively, there is another network model  $\MGplus \maps \S \to \Mon$ with 
     values $(\MG(n), +)$ where $+$  is multiset sum. That is, $g + h$ obtained by adding multiplicities of corresponding edges. 
\end{ex}
 
\begin{ex}[\textbf{Directed multigraphs}] \label{ex:directed_multigraph}
    Let a \define{directed multigraph} on a set $V$ be a multiset of ordered pairs $(i,j) \in V^2$ such that $i \ne j$. There is a network model $\DMG \maps \S \to \Mon$ such that $\DMG(n)$ is the set of directed multigraphs on $\n$ with monoid operation the union of multisets. Alternatively, there is a network model with values $(\DMG(n), +)$ where $+$  is multiset sum.
 \end{ex}
 
\begin{ex}[\textbf{Hypergraphs}]
 \label{ex:hypergraph}
   Let a \define{hypergraph} on a set $V$ be a set of nonempty subsets of $V$, called \define{hyperedges}. There is a network model $\HG \maps \S \to \Mon$ such that $\HG(n)$ is the set of hypergraphs on $\n$. The monoid operation $\HG(n)$ is union.
\end{ex}

\begin{ex}[\textbf{Graphs with colored edges}]
\label{ex:graphs_with_colored_edges}
    Fix a set $B$ of \define{edge colors} and let $\SG \maps \S \to \Mon$ be the network model of simple graphs as in Ex.\ \ref{ex:simple_graph}. Then there is a network model $H \maps \S \to \Mon$ with
    \[
        H(n) = \SG(n)^B
    \]
    making the product of $B$ copies of the monoid $\SG(n)$ into a monoid in the usual way. In this model, a network is a $B$-tuple of simple graphs, which we may view as a graph with at most one edge of each color between any pair of distinct vertices. We describe this construction in more detail in Ex.\ \ref{ex:graphs_with_colored_edges_2}.
\end{ex}

There are also examples of network models not involving graphs:

\begin{ex}[\textbf{Partitions}]
\label{ex:partitions}
    A poset is a lattice if every finite subset has both an infimum and a supremum. If $L$ is a lattice, then $(L, \vee)$ and $(L, \wedge)$ are both monoids, where $x \vee y$ is the supremum of $\{x,y\} \subseteq L$ and $x \wedge y$ is the infimum. 
    
    Let $P(n)$ be the set of partitions of the set $\n$. This is a lattice where $\pi \le \pi'$ if the partition $\pi$ is finer than $\pi'$. Thus, $P(n)$ can be made a monoid in either of the two ways mentioned above. Denote these monoids as $P^{\vee}(n)$ and $P^{\wedge}(n)$. These monoids extend to give two network models $P^{\vee}, P^{\wedge} \maps \S \to \Mon$. 
\end{ex}
 
\section{One-colored network models from monoids}
\label{sec:models_from_monoids}

There is a systematic procedure that gives many of the network models we have seen so far. To do this, we take networks to be ways of labelling the edges of a complete graph by elements of some monoid $M$. The operation of overlaying two of these networks is then described using the monoid operation. 

For example, consider the Boolean monoid $\Boole$: that is, the set $\{F,T\}$ with `inclusive or' as its monoid operation. A complete graph with edges labelled by elements of $\Boole$ can be seen as a simple graph if we let $T$ indicate the presence of an edge between two vertices and $F$ the absence of an edge. To overlay two simple graphs $g_1, g_2$ with the same set of vertices we simply take the `or' of their edge labels. This gives our first example of a network model, Ex.\ \ref{ex:simple_graph}. 

To formalize this we need some definitions. Given $n \in \N$, let $\E(n)$ be the set of 2-element subsets of $\n = \{1, \dots, n\}$. We call the members of $\E(n)$ \define{edges}, since they correspond to edges of the complete graph on the set $\n$. We call the elements of an edge $e \in \E(n)$ its \define{vertices}.

Let $M$ be a monoid. For $n \in \N$, let $\G_M(n)$ be the set of functions $g \maps \E(n) \to M$. Define the operation $\cup \colon \G_M(n) \times \G_M(n) \to \G_M(n)$ by $(g_1 \cup g_2)(e) = g_1(e) g_2(e)$ for $e \in \E(n)$.  Define the map $\sqcup \colon \G_M(m) \times \G_M(n) \to \G_M(m+n)$ by
\[
    (g_1 \sqcup g_2)(e) =
    \left\{\begin{array}{cl}
    g_1(e) & {\rm if \; both \; vertices \; of \;} e {\rm \; are \;}\leq m 
    \\
    g_2(e) & {\rm if \; both \; vertices \; of \;} e {\rm \; are \;} > m %
    \\
    {\rm the \; identity \; of \; } M & {\rm otherwise} 
    \\
    \end{array}
    \right.
\]
The symmetric group $S_n$ acts on $\G_M(n)$ by $\sigma(g)(e) = g(\sigma^{-1}(e))$.

\begin{thm}
\label{thm:graph_model}
    For each monoid $M$ the data above gives a one-colored network model $\G_M \maps \S \to \Mon$.
\end{thm}
\begin{proof}
We can define $\G_M$ as the composite of two functors, $\E \maps \S \to \Inj$ and $M^{-} \maps \Inj \to \Mon$, where $\Inj$ is the category of sets and injections.

The functor $\E \maps \S \to \Inj$ sends each object $n \in \S$ to $\E(n)$, and it sends each morphism $\sigma \maps n \to n$ to the permutation of $\E(n)$ that maps any edge $e = \{x,y\} \in \E(n)$ to $\sigma(e) = \{\sigma(x), \sigma(y)\}$. The category $\Inj$ does not have coproducts, but it is closed under coproducts in $\Set$. It thus becomes symmetric monoidal with $+$ as its tensor product and the empty set as the unit object. For any $m, n \in \S$ there is an injection
\[\mu_{m,n} \maps \E(m) + \E(n) \to \E(m+n) \]
expressing the fact that a 2-element subset of either $\m$ or $\n$ gives a 2-element subset of $\m+\n$. The functor $\E \maps \S \to \Inj$ becomes lax symmetric monoidal with these maps $\mu_{m,n}$ giving the lax preservation of the tensor product.

The functor $M^- \maps \Inj \to \Mon$ sends each set $X$ to the set $M^X$ made into a monoid with pointwise operations, and it sends each function $f \maps X \to Y$ to the monoid homomorphism $M^f \maps M^X \to M^Y$ given by
\[(M^f g)(y) = \left\{ \begin{array}{ccl}
g(f^{-1}(y)) & \textrm{if } y \in \mathrm{im}(f) \\
1 & \textrm{otherwise} 
\end{array} \right.\]
for any $g \in M^X$. Using the natural isomorphisms $M^{X + Y} \cong M^X \times M^Y$ and $M^{\emptyset} \cong 1$ this functor can be made symmetric monoidal. 

As the composite of the lax symmetric monoidal functor $\E \maps \S \to \Inj$ and the symmetric monoidal functor $M^- \maps \Inj \to \Mon$, the functor $\G_M \maps \S \to \Mon$ is lax symmetric monoidal, and thus a network model. With the help of Thm.\ \ref{thm:equations}, it is easy to check that this description of $\Gamma_M$ is equivalent to that in the theorem statement.
\end{proof}

\begin{ex}[\textbf{Simple graphs, revisited}]
\label{ex:simple_graph_2}
    Let $\Boole = \{F,T\}$ be the Boolean monoid. If we interpret $T$ and $F$ as `edge' and `no edge' respectively, then $\G_{\Boole}$ is just $\SG$, the network model of simple graphs discussed in Example \ref{ex:simple_graph}.
\end{ex}

Recall from Ex.\ \ref{ex:multigraph} that a multigraph on the set $\n$ is a multisubset of $\E(n)$, or in other words, a function $g \maps \E(n) \to \N$. There are many ways to create a network model $F \maps \S \to \Mon$ for which $F(n)$ is the set of multigraphs on the set $\n$, since $\N$ has many monoid structures. Two of the most important are these:

\begin{ex}[\textbf{Multigraphs with addition for overlaying}]
\label{ex:multigraph_2}
    Let $(\N, +)$ be $\N$ made into a monoid with the usual notion of addition as $+$. In this network model, overlaying two multigraphs $g_1, g_2 \maps \E(n) \to \N$ gives a multigraph $g \maps \E(n) \to \N$ with $g(e) = g_1(e) + g_2(e)$. In fact, this notion of overlay corresponds to forming the multiset sum of edge multisets and $\G_{(\N,+)}$ is the network model of multigraphs called $\MGplus$ in Ex.\ \ref{ex:multigraph}. 
\end{ex}

\begin{ex}[\textbf{Multigraphs with maximum for overlaying}]
\label{ex:multigraph_3}
    Let $(\N, \max)$ be $\N$ made into a monoid with $\max$ as the monoid operation. Then $\G_{(\N,\max)}$ is a network model where overlaying two multigraphs $g_1, g_2 \maps \E(n) \to \N$ gives a multigraph $g \maps \E(n) \to \N$ with $g(e) = g_1(e) \max g_2(e)$.
    For this monoid structure overlaying two copies of the same multigraph gives the same multigraph. In other words, every element in each monoid $\G_{(\N,\max)}(n)$ is idempotent and $\G_{(\N,\max)}$ is the network model of multigraphs called $\MG$ in Ex.\ \ref{ex:multigraph}. 
\end{ex}

\begin{ex}[\textbf{Multigraphs with at most $k$ edges between vertices}]
\label{ex:multigraph_with_at_most_k}
For any $k \in \N$, let $\Boole_k$ be the set $\{0,\dots,k\}$ made into a monoid with the monoid operation $\oplus$ given by 
\[x \oplus y = (x + y) \min k \]
and $0$ as its unit element. For example, $\Boole_0$ is the trivial monoid and $\Boole_1$ is isomorphic to the Boolean monoid. There is a network model $\G_{\Boole_k}$ such that $\G_{\Boole_k}(n)$ is the set of multigraphs on $\n$ with at most $k$ edges between any two distinct vertices. 
\end{ex}

\section{Network models}
\label{sec:netmod_C}

The network models described so far allow us to handle graphs with colored edges, but not with colored vertices. Colored vertices are extremely important for applications in which we have a network of agents of different types. Thus, network models will involve a set $C$ of vertex colors in general. This requires that we replace $\S$ by the free strict symmetric monoidal category generated by the color set $C$. Thus, we begin by recalling this category.

For any set $C$, there is a category $\SC$ for which:
\begin{itemize}
    \item Objects are formal expressions of the form
    \[
        c_1 \otimes \cdots \otimes c_n 
    \]
    for $n \in \N$ and $c_1, \dots, c_n \in C$. 
    We denote the unique object with $n = 0$ as $I$.
    \item There exist morphisms from $c_1 \otimes \cdots \otimes c_m$ to $c'_1 \otimes \cdots \otimes c'_n$ only if $m = n$, and in that case a morphism is a permutation $\sigma \in S_n$ such that $c'_{\sigma(i)} = c_i$ for all $i$.
    \item Composition is the usual composition of permutations. 
\end{itemize}

Note that elements of $C$ can be identified with certain objects of $\S(C)$, namely the one-fold tensor products. We do this in what follows.

\begin{prop}
\label{prop:free}
    $\S(C)$ can be given the structure of a strict symmetric monoidal category making it into the free strict symmetric monoidal category on the set $C$. Thus, if $\A$ is any strict symmetric monoidal category and $f \maps C \to \Ob(\A)$ is any function from $C$ to objects of the $\A$, there exists a unique strict symmetric monoidal functor $F \maps \S(C) \to \A$ with $F(c) = f(c)$ for all $c \in C$.
\end{prop}

\begin{proof}
This is well-known; see for example Sassone \cite[Sec.\ 3]{Sassone} or Gambino and Joyal \cite[Sec.\ 3.1]{GJ}. The tensor product of objects is $\otimes$, the unit for the tensor product is $I$, and the braiding 
\[(c_1 \otimes \cdots \otimes c_m) \otimes (c'_1 \otimes \cdots \otimes c'_n) \to (c'_1 \otimes \cdots \otimes c'_n)  \otimes (c_1 \otimes \cdots \otimes c_m) \]
is the block permutation $B_{m,n}$. Given $f \maps C \to \Ob(\A)$, we define $F\maps \S(C) \to \A$ on objects by
 \[F(c_1 \otimes \cdots \otimes c_n) = f(c_1) \otimes \cdots \otimes f(c_n) ,\]
and it is easy to check that $F$ is strict symmmetric monoidal, and the unique functor with the required properties.
\end{proof}

\begin{defn}
\label{defn:colored_network_model}
    Let $C$ be a set, called the set of \define{vertex colors}. A
    $C$\define{-colored network model} is a lax symmetric monoidal functor 
     \[F \maps  \SC \to \Cat. \] 
    A \define{network model} is a $C$-colored network model for some set $C$.
\end{defn}

If $C$ has just one element, $\S(C) \cong \S$ and a $C$-colored network model is a one-colored network model in the sense of Def.\ \ref{defn:network_model}. Here are some more interesting examples:

\begin{ex}[\textbf{Simple graphs with colored vertices}]
\label{ex:simple_graphs_with_colored_vertices}
    There is a network model of simple graphs with $C$-colored vertices. To construct this, we start with the network model of simple graphs $\SG \maps \S \to \Mon$ given in Ex.\ \ref{ex:simple_graph}. There is a unique function from $C$ to the one-element set. By Prop.\ \ref{prop:free}, this function extends uniquely to a strict symmetric monoidal functor 
    \[F \maps \S(C) \to \S . \]
    An object in $\S(C)$ is formal tensor product of $n$ colors in $C$; applying $F$ to this object we forget the colors and obtain the object $n \in \S$. Composing $F$ and $\SG$, we obtain a lax symmetric monoidal functor
    \[
        \S(C) \stackrel{F}{\longrightarrow} \S  \stackrel{\SG}{\longrightarrow} \Mon
    \]
    which is the desired network model. We can use the same idea to `color' any of the network models in Section \ref{sec:netmod}.
    
    Alternatively, suppose we want a network model of simple graphs with $C$-colored vertices where an edge can only connect two vertices of the same color. For this we take a cartesian product of $C$ copies of the functor $\SG$, obtaining a lax symmetric monoidal functor 
    \[{\SG}^C \maps \S^C \to \Mon^C. \]  
    There is a function $h \maps C \to \Ob(\S^C)$ sending each $c \in C$ to the object of $S^\C$ that equals $1 \in \S$ in the $c$th place and $0 \in \S$ elsewhere. Thus, by Prop.\ \ref{prop:free}, $h$ extends uniquely to a strict symmetric monoidal functor 
    \[H_C \maps \S(C) \to \S^C  .\]
    Furthermore, the product in $\Mon$ gives a symmetric monoidal functor
    \[\Pi \maps \Mon^C \to \Mon .\]
    Composing all these, we obtain a lax symmetric monoidal functor
    \[
        \SC \stackrel{H_C}{\longrightarrow} \S^C \stackrel{\SG^C}{\longrightarrow} \Mon^C \stackrel{\Pi}{\longrightarrow} \Mon
    \]
    which is the desired network model. 
    
    More generally, if we have a network model $F_c \maps \S \to \Mon$ for each color $c \in C$, we can use the same idea to create a network model:
    \[
        \xymatrix{
        \SC \ar[r]^{H_C} 
        & \S^C  \ar[rr]^{\prod_{c \in C} F_c} & &
        \Mon^C \ar[r]^{\Pi} & \Mon
        }
    \]
    in which the vertices of color $c \in C$ partake in a network of type $F_c$.
    \label{ex:colors}
\end{ex}

\begin{ex}[\textbf{Petri nets}]
    Petri nets are a kind of network widely used in computer science, chemistry and other disciplines \cite{RxNet}. A \define{Petri net} $(S, T, i, o)$ is a pair of finite sets and a pair of functions $i, o \maps S \times T \to \N$. Let $P(m,n)$ be the set of Petri nets $(\m, \n, i, o)$. This becomes a monoid with product
    \[
        (\m, \n, i, o) \cup (\m, \n, i', o')
        = (\m, \n, i+i', o+o')
    \]
    The groups $S_m\times S_n$ naturally act on these monoids, so we have a functor 
    \[P \maps \S^2 \to \Mon . \]
    There are also `disjoint union' operations
    \[\sqcup \maps P(m,n) \times P(m',n') \to P(m+m', n+n') \]
    making $P$ into a lax symmetric monoidal functor. In Ex.\ \ref{ex:simple_graphs_with_colored_vertices} we described a strict symmetric monoidal functor $H_C \maps \S(C) \to \S^C$ for any set $C$. In the case of the 2-element set this gives
    \[H_2 \maps \S(2) \to \S^2 .\]
    We define the network model of Petri nets to be the composite
    \[\S(2) \stackrel{H_2}{\longrightarrow} \S^2 \stackrel{P}{\longrightarrow} \Mon .\]
\end{ex}

\section{Categories of network models}
\label{sec:cat_netmod}

For each choice of the set $C$ of vertex colors, we can define a category $\NM_C$ of $C$-colored network models. However, it is useful to create a larger category $\NM$ containing all these as subcategories, since there are important maps between network models that involve changing the vertex colors. 

\begin{defn}
\label{defn:NM_C}
For any set $C$, let $\NM_C$ be the category for which:
\begin{itemize}
    \item an object is a $C$-colored network model, that is, a lax symmetric monoidal functor $F \maps \SC\to \Cat$,
    \item a morphism is a monoidal natural transformation between such functors:
    \[\begin{tikzcd}
    \SC
    \arrow[r,"F", bend left=40]
    \arrow[bend left=40]{r}[name=LUU, below]{}
    \arrow[r,"F'", bend right=40,swap, pos=0.45]
    \arrow[bend right=40, pos = 0.53]{r}[name=LDD]{}
    \arrow[Rightarrow,to path=(LUU) -- (LDD)\tikztonodes]{r}{\gn}
    & 
    \Cat
    \end{tikzcd}\]
    and composition is the usual composition of monoidal natural transformations.
    \end{itemize}
\end{defn}

In particular, $\NM_1$ is the category of one-colored network models. For an example involving this category, consider the network models built from monoids in Sec.\ \ref{sec:models_from_monoids}. Any monoid $M$ gives a one-colored network model $\G_M$ for which an element of $\G_M(n)$ is a way of labelling the edges of the complete graph on $\n$ by elements of $M$. Thus, we should expect any homomorphism of monoids $f \maps M \to M'$ to give a morphism of network models $\G_f \maps \G_M \to \G_{M'}$ for which 
\[\G_f(n) \maps \G_M(n) \to \G_{M'}(n)  \]
applies $f$ to each edge label. 

Indeed, this is the case. As explained in the proof of Thm.\ \ref{thm:graph_model}, the network model $\G_M$ is the composite
\[\S \stackrel{\E}{\longrightarrow} \Inj 
\stackrel{M^{-}}{\longrightarrow} \Mon .\]
The homomorphism $f$ gives a natural transformation
\[f^{-} \maps M^{-} \To M'^{-}  \]
that assigns to any finite set $X$ the monoid homomorphism
\[\begin{array}{rccl} 
f^X \maps & M^X & \to     & M'^X  \\
          &  g  & \mapsto & f \circ g .
\end{array}
\]      
It is easy to check that this natural transformation is monoidal. Thus, we can whisker it with the lax symmetric monoidal functor $\E$ to get a morphism of network models:
\[\begin{tikzcd}
    \S \arrow[r,"\E"] &
    \Inj
    \arrow[r,"M^-", bend left=40]
    \arrow[bend left=40]{r}[name=LUU, below]{}
    \arrow[r,"M'^-", bend right=40,swap, pos=0.45]
    \arrow[bend right=40]{r}[name=LDD]{}
    \arrow[Rightarrow,to path=(LUU) -- (LDD)\tikztonodes]{r}{f^-}
    & 
    \Mon
    \end{tikzcd}\]
and we call this $\G_f \maps \G_M \to \G_{M'}$.

\begin{thm}
\label{thm:models_from_monoids}
There is a functor 
\[\G \maps \Mon \to \NM_1  \]
sending any monoid $M$ to the network model $\G_M$ and any homomorphism of monoids $f \maps M \to M'$ to the morphism of network models $\G_f \maps \G_M \to \G_{M'}$.
\end{thm}

\begin{proof}
To check that $\G$ preserves composition, note that
\[\begin{tikzcd}[column sep=huge]
    \S 
    \arrow[r,"\E"] 
    &
    \Inj
    \arrow[r, "M^-", bend left=80]
    \arrow[r, ""{name=TOP}, bend left=80, swap, pos=0.455]
    \arrow[r, "M'^-"{name=Ml}] 
    \arrow[Rightarrow, from=TOP, to=Ml, "f^-", pos=0.3]
    \arrow[r, ""{name=M}, swap]
    \arrow[r, "M''^-", bend right=80, swap]
    \arrow[r, ""{name=BOT}, bend right=80, pos=0.45]
    \arrow[Rightarrow, from=M, to=BOT, "f'^-", pos=0.5]
    & 
    \Mon
\end{tikzcd}\]
equals
\[\begin{tikzcd}[column sep=huge]
    \S \arrow[r,"\E"] &
    \Inj
    \arrow[r,"M^-", bend left=80]
    \arrow[bend left=80,pos=0.47]{r}[name=LUU, below]{}
    \arrow[r,"M''^-", bend right=80,swap, pos=0.5]
    \arrow[bend right=80,pos=0.44]{r}[name=LDD]{}
    \arrow[Rightarrow, from=LUU, to=LDD, "(f'f)^-"]
    & 
    \Mon
    \end{tikzcd}   \]
since $f'^- f^- = (f'f)^-$. Similarly $\G$ preserves identities. \end{proof}

It has been said that category theory is the subject in which even the examples need examples. So, we give an example of the above result:

\begin{ex}[\textbf{Imposing a cutoff on the number of edges}]
\label{ex:cutoff}
In Ex.\ \ref{ex:multigraph_2} we described the network model of multigraphs $\MGplus$ as $\G_{(\N,+)}$. In Ex.\ \ref{ex:multigraph_with_at_most_k} we described a network model $\G_{\Boole_k}$ of multigraphs with at most $k$ edges between any two distinct vertices. There is a homomorphism of monoids
\[\begin{array}{rccl}     f \maps &(\N,+) &\to& \Boole_k \\
& n &\mapsto & n \min k
\end{array} \]
and this induces a morphism of network models
\[\G_f \maps \G_{(\N,+)} \to \G_{\Boole_k} .\]
This morphism imposes a cutoff on the number of edges between any two distinct vertices: if there are more than $k$, this morphism keeps only $k$ of them. In particular, if $k = 1$, $\Boole_k$ is the Boolean monoid, and 
\[\G_f \maps \MGplus \to \SG \]
sends any multigraph to the corresponding simple graph.
\end{ex}

One useful way to combine $C$-colored networks is by `tensoring' them. This makes $\NM_C$ into a symmetric monoidal category:

\begin{thm}
\label{thm:tensor_product_of_network_models}
For any set $C$, the category $\NM_C$ can be made into a symmetric monoidal category with the tensor product defined pointwise, so that for objects $F, F' \in \NM_C$ we have 
\[(F \otimes F')(x) = F(x) \times F'(x) \]
for any object or morphism $x$ in $\S(C)$, and for morphisms $\phi, \phi'$ in $\NM_C$ we have
\[(\phi \otimes \phi')_x = \phi_x \times \phi'_x \]
for any object $x \in \S(C)$.
\end{thm}

\begin{proof}
    More generally, for any symmetric monoidal categories $\A$ and $\B$, there is a symmetric monoidal category $\hom_{\SMC}(\A,\B)$ whose objects are lax symmetric monoidal functors from $\A$ to $\B$ and whose morphisms are monoidal natural transformations, with the tensor product defined pointwise. The proof in the `weak' case was given by Hyland and Power \cite{HP}, and the lax case works the same way.
\end{proof}

If $F,F' \maps \S(C) \to \Mon$ then their tensor product again takes values in $\Mon$. There are many interesting examples of this kind:

\begin{ex}[\textbf{Graphs with colored edges, revisited}]
\label{ex:graphs_with_colored_edges_2}
    In Ex.\ \ref{ex:graphs_with_colored_edges} we described network models of simple graphs with colored edges. The above result lets us build these network models starting from more basic data. To do this we start with the network model for simple graphs, $\SG \maps \S \to \Mon$, discussed in Ex.\ \ref{ex:simple_graph}. Fixing a set $B$ of `edge colors', we then take a tensor product of copies of $\SG$, one for each $b \in B$. The result is a network model $\SG^{\otimes B} \maps \S \to \Mon$ with 
    \[\SG^{\otimes B}(n) = \SG(n)^B  \]
    for each $n \in \N$.
\end{ex}

\begin{ex}[\textbf{Combined networks}]
\label{ex:mixed_networks}
    We can also combine networks of different kinds. For example, if $\DG \maps \S \to \Mon$ is the network model of directed graphs given in Ex.\ \ref{ex:directed_graph} and $\MG \maps \S \to \Mon$ is the network model of multigraphs given in Ex.\ \ref{ex:multigraph}, then   
    \[\DG \otimes \MG \maps \S \to \Mon \]
    is another network model, and we can think of an element of $(\DG \otimes \MG)(n)$ as a directed graph with red edges together with a multigraph with blue edges on the set $\n$.
\end{ex}

Next we describe a category $\NM$ of network models with arbitrary color sets, which includes all the categories $\NM_C$ as subcategories. To do this, first we introduce `color-changing' functors. Recall that elements of $C$ can be seen as certain objects of $\S(C)$, namely the 1-fold tensor products. If $f \maps C \to C'$ is a function, there exists a unique strict symmetric monoidal functor $f_* \maps \S(C) \to \S(C')$ that equals $f$ on objects of the form $c \in C$. This follows from Prop.\ \ref{prop:free}.

Next, we define an indexed category $\NM_{-} \maps \Set^{\op} \to \CAT$ that sends any set $C$ to $\NM_C$ and any function $f \maps C \to D$ to the functor that sends any $D$-colored network model $F \maps \S(D) \to \Cat$ to the $C$-colored network model given by the composite
\[\S(C) \xrightarrow{f_*} \S(D) \xrightarrow{F} \Cat .\]
Applying the Grothendieck construction to this indexed category, we define the category of network models to be
\[\NM = \int \NM_-. \]
In elementary terms, $\NM$ has:
\begin{itemize}
    \item pairs $(C,F)$ for objects, where $C$ is a set and $F \maps \S(C) \to \Cat$ is a $C$-colored network model.
    \item pairs $(f,g) \maps (C,F) \to (D,G)$ for morphisms, where $f \maps C \to D$ is a function and $g \maps F \Rightarrow G \circ f_*$ is a morphism of network models.
\end{itemize}

\begin{ex}[\textbf{Simple graphs with colored vertices, revisited}]
    In Ex.\ \ref{ex:simple_graphs_with_colored_vertices} we constructed the network model of simple graphs with colored vertices. We started with the network model for simple graphs, which is a one-colored network model $\SG \maps \S \to \Mon$. The unique function $! \maps C \to 1$ gives a strict symmetric monoidal functor $!_* \maps \S(C) \to \S(1) \cong \S$. The network model of simple graphs with $C$-colored vertices is the composite 
     \[
       \S(C) \stackrel{!_*}{\longrightarrow} \S \stackrel{\SG}{\longrightarrow} \Mon
    \]
    and there is a morphism from this to the network model of simple graphs, which has the effect of forgetting the vertex colors.
\end{ex}

In fact, $\NM$ can be understood as a subcategory of the following category:

\begin{defn} 
\label{defn:SMICat}
Let $\SMICat$ be the category where:
\begin{itemize}
    \item objects are pairs $(\C,F)$ where $\C$ is a small symmetric monoidal category and $F \maps \C \to \Cat$ is a lax symmetric monoidal functor, where $\Cat$ is considered with its cartesian monoidal structure.
    \item morphisms from $(\C,F)$ to $(\C',F')$ are pairs $(G,\gn)$ where $G \maps \C \to \C'$ is a lax symmetric monoidal functor and $\gn \maps F \To F' \circ G $ is a symmetric monoidal natural transformation:
    \[\begin{tikzcd}
    \C
    \arrow[dr, "F"]
    \arrow[dr, ""{name=F}, swap]
    \arrow[dd, "G", swap]
    \\&
    \Cat
    \\
    \C'
    \arrow[ur, "F'", swap]
    \arrow[ur, ""{name=F'}, pos=0.43]
    \arrow[Rightarrow, from = F, to = F', "\gn", swap]
\end{tikzcd}\]
    \end{itemize}
\end{defn}

We shall use this way of thinking in the next two sections to build operads from network models. For experts, it is worth admitting that $\SMICat$ is part of a 2-category where a 2-morphism $\xi \maps (G,\gn) \To (G',\gn')$ is a natural transformation $\xi \maps G \to G'$ such that 
\[\begin{tikzcd}
    \C 
    \arrow[dd, bend right=90, "G", pos=0.495, swap]
    \arrow[dd, bend right=90, ""{name=L, right}, swap]
    \arrow[dd,"G'"{name=R, left}, swap]
    \arrow[dr,"F", pos=0.4]
    \arrow[dr,""{name=U, below}, pos=0.4]
    & & &  
    \C'
    \arrow[dd, "G"{name=R2, left}, pos=0.52, swap]
    \arrow[dr,"F", pos=0.4]
    \arrow[dr,""{name=U2, below}, pos=0.44]
    \\
    & \Cat & = & &  \Cat.
    \\
    \C
    \arrow[ur,"F'",swap, pos=0.4]
    \arrow[ur,""{name=D}, pos=0.46]
    \arrow[Rightarrow, from=U, to=D, "\gn'", swap]
    \arrow[Rightarrow, from=L, to=R, "\xi"{above}, swap]
    & & & 
    \C'
    \arrow[ur,"F'",swap, pos=0.4]
    \arrow[ur,""{name=D2}]
    \arrow[Rightarrow, from=U2, to=D2, "\gn", swap, pos=0.55]
    \end{tikzcd}
    \]
This lets us define 2-morphisms between network models, extending $\NM$ to a 2-category. We do not seem to need these 2-morphisms in our applications, so we suppress 2-categorical considerations in most of what follows. However, we would not be surprised if the 2-categorical aspects of network models turn out to be important, and we touch on them in Sec.\ \ref{subsec:mon_fib}.

\section{The Grothendieck construction}
\label{sec:Grothendieck}

In this section we describe how to build symmetric monoidal categories using the Grothendieck construction. In the next section we use this to construct operads from network models, but the material here is self-contained and of independent interest. 

In what follows we always give $\Cat$ its cartesian symmetric monoidal structure.  Given a small category $\C$ and a functor $F \maps \C \to \Cat$, the \define{Grothendieck construction} gives a category $\Int F$ where:
\begin{itemize}
    \item the objects are pairs $(c, x)$, where $c\in \C$ and $x$ is an object in $Fc$;
    \item the morphisms are $(f, g) \maps (c,x) \to (d,y)$ where $f \maps c\to d$ is a morphism in $\C$ and $g \maps Ff(x) \to y$ is a morphism in $Fd$;
    \item the composite of
    \[
        (f',g') \maps (c,x)\to (d,y)
    \]
    and
    \[
        (f,g) \maps (d,y) \to (e,z)
    \]
    is given by
    \begin{equation}
    \label{eq:composition}
        (f, g) \circ (f', g') = (f \circ f', g \circ F(f)(g')).
    \end{equation}
\end{itemize}

In what follows we prove:
\begin{itemize}
    \item \textbf{Thm.\ \ref{thm:G_construction_for_mon_cats}}: if $\C$ is monoidal and $F$ is lax monoidal, then $\Int F$ is monoidal category.
    \item \textbf{Thm.\ \ref{thm:G_construction_for_braided_cats}}: If $\C$ is braided monoidal and $F$ is lax braided monoidal, then $\Int F$ is braided monoidal.
    \item \textbf{Thm.\ \ref{thm:G_construction_functorial_for_symmetric_cats}}: If $\C$ is symmetric monoidal and $F$ is lax symmetric monoidal, then $\Int F$ is symmetric monoidal.
\end{itemize}
Moreover, in each case the Grothendieck construction is functorial. For example, in Def.\ \ref{defn:SMICat} we described a category $\SMICat$ where an object is a symmetric monoidal category $\C$ equipped with a lax symmetric monoidal functor $F \maps \C \to \Cat$. In Thm.\ \ref{thm:G_construction_functorial_for_symmetric_cats} we show that the Grothendieck construction gives a functor
\[\textstyle{\Int} \maps \SMICat \; \to \; \SMC .\]
In proving these results we take a self-contained and elementary approach which provides the equations that we need later. We sketch a more high-powered 2-categorical approach using fibrations and indexed categories in Sec.\ \ref{subsec:mon_fib}.

We use the following lemma implicitly whenever we need to construct an isomorphism in $\Int F$:

\begin{lem}
    If $f \maps c \to d$ and $g \maps F(f)(x) \to y$ are isomorphisms, then $(f,g)$ is an isomorphism in $\Int F$. 
\end{lem}
\begin{proof}
    Na\"ively one might think that $(f \inv, g \inv)$ should be the inverse of $(f,g)$. However, $(f\inv,g\inv)$ is not even a morphism from $(d,y)$ to $(c,x)$ since $g\inv$ does not go from $Ff\inv(y)$ to $x$. The inverse of $(f,g)$ is $(f\inv, Ff\inv g\inv)$:
    \begin{align*}
        (f, g) \circ (f\inv, Ff\inv g\inv)
        &= (f\circ f\inv, g \circ (F(f))(F(f\inv) g\inv))
        \\&= (1_c, g \circ (F(f)\circ F(f\inv)) (g\inv))
        \\&= (1_c, g \circ g\inv)
        \\&= (1_c, 1_x)
    \end{align*}
    \begin{align*}
        (f\inv, F(f\inv) g\inv) \circ (f,g)
        &= (f\inv \circ f, Ff\inv (g\inv) \circ Ff\inv(g))
        \\&= (1_d, Ff\inv (g\inv \circ g))
        \\&= (1_d, 1_y) \qedhere
    \end{align*}
\end{proof}

Next we discuss the functoriality of the Grothendieck construction. 

\begin{defn} 
\label{defn:ICat}
    Let $\ICat$ denote the category where
    \begin{itemize}
        \item an object is a pair $(\C,F)$ where $\C$ is a small category and $F \maps \C \to \Cat$ is a functor
        \item a morphism $(\C_1, F_1) \to (\C_2, F_2)$ is a pair $(G,\gn)$ where $G \maps \C_1 \to \C_2$ is a functor and $\gn \maps F_1 \To F_2 \circ G$ is a natural transformation:
        \[\begin{tikzcd}
            \C_1
            \arrow[dr, "F_1"]
            \arrow[dr, ""{name=F}, swap]
            \arrow[dd, "G", swap]
            \\&
            \Cat
            \\
            \C_2
            \arrow[ur, "F_2", swap]
            \arrow[ur, ""{name=F'}, pos=0.45]
            \arrow[Rightarrow, from = F, to = F', "\gn", swap]
        \end{tikzcd}\]
    \end{itemize}
    For brevity, we denote the object $(\C, F)$ as simply $F$, and the morphism $(G,\gn)$ as simply $G$. 
\end{defn}

Having defined the Grothendieck construction on objects of $\ICat$, we proceed to define it on morphisms. Suppose we have a morphism in $\ICat$:
\[
\begin{tikzcd}
    \C
    \arrow[dr, "F"]
    \arrow[dr, ""{name=F}, swap]
    \arrow[dd, "G", swap]
    \\&
    \Cat
    \\
    \C'
    \arrow[ur, "F'", swap]
    \arrow[ur, ""{name=F'}, pos=0.43]
    \arrow[Rightarrow, from = F, to = F', "\gn", swap]
\end{tikzcd}\]
Then we can define a functor $\Ghat \maps \Int F \to \Int F'$ as follows.
\[
\begin{tikzcd}
    &(c,x)
    \arrow[dd, "{(f,g)}", swap]
    &&
    (Gc, \gn_c x)
    \arrow[dd, "{(Gf, \gn_d g)}"]
    \\
    \Ghat \maps
    && 
    \mapsto
    \\
    &(d,y)
    &&
    (Gd, \gn_d y)
\end{tikzcd}\]

The following result is well-known \cite{Borceux}:
\begin{thm}
    \label{thm:G_construction_functorial}
    There exists a unique functor, the \define{Grothendieck construction} 
    \[
        \textstyle{\Int} \maps \ICat \; \to \; \Cat
    \]
    sending any object $F$ to the category $\Int F$ and sending any morphism $G \maps F \to F'$ to the functor $\Ghat \maps \Int F \to \Int F'$.
\end{thm}

\subsection{The monoidal Grothendieck construction}

Next we explain how to use the Grothendieck construction to build monoidal categories.

\begin{defn}
\label{defn:MC}
    Let $\MC$ be the category with small monoidal categories as objects and lax monoidal functors as morphisms.
\end{defn}

\begin{defn} 
\label{defn:MICat}
Let $\MICat$ be the category of \emph{lax monoidal functors into} $\Cat$, where:
\begin{itemize}
    \item objects are pairs $(\C,F)$ where $\C$ is a small monoidal category and $F \maps \C \to \Cat$ is a lax monoidal functor.
    \item morphisms from $(\C_1, F_1)$ to $(\C_2,G_2)$ 
    are pairs $(G,\gn)$ where $G \maps \C_1 \to \C_2$ is a lax monoidal functor and $\gn \maps F \To F' \circ G $ is a monoidal natural transformation:
    \[\begin{tikzcd}
    \C_1
    \arrow[dr, "F_1"]
    \arrow[dr, ""{name=F}, swap]
    \arrow[dd, "G", swap]
    \\&
    \Cat
    \\
    \C_2
    \arrow[ur, "F_2", swap]
    \arrow[ur, ""{name=F'}, pos=0.45]
    \arrow[Rightarrow, from = F, to = F', "\gn", swap]
    \end{tikzcd}\]
    \end{itemize}
\end{defn}

Our goal in this subsection is to refine the Grothendieck construction to a functor
\[\textstyle{\Int} \maps \MICat \; \to \MC .\]
Given $\C$ a monoidal category, and $F \maps \C \to \Cat$ a lax monoidal functor, we define a monoidal structure on $\Int F$. This construction makes use of every aspect of the monoidal structure on $\C$: the functor $\otimes \maps \C \times C \to \C$, the unit object $I \in \C$, and the natural isomorphisms $\alpha_{c,d,e} \maps (c \otimes d) \otimes e \to c \otimes (d \otimes e)$, $\lambda_c \maps I \otimes c \to c$, and $\rho_c \maps c \otimes I \to c$. It also uses the lax monoidal structure of $F$: the natural transformation $\Phi_{c,c'} \maps Fc \times Fc' \to F (c \otimes c')$ and the morphism $\phi \maps I_{\Cat} \to I$. 

First, given two objects $(c,x)$ and $(c',x')$ of $\Int F$, we define their tensor product by
\[(c,x) \otimes_F (c',x') = ( c\otimes c', \Phi_{c,c'}(x,x') ).\]
Next, consider two morphisms 
\begin{align*} 
    (f,g) &\maps (c,x) \to (d,y)
    \\(f',g') &\maps (c',x') \to (d',y')
\end{align*}
in $\Int F$. We take the first component of $(f,g) \otimes (f',g') \maps  (c\otimes c', \Phi_{c,c'}(x,x') ) \to (d \otimes d', \Phi_{d,d'}(y,y') )$ to be $f \otimes f'$. The second component must then be a morphism from $F (f \otimes f')( \Phi_{c,c'} (x,x') )$ to $\Phi_{d,d'}(y,y')$. To meet this condition we define the tensor product of morphisms in $\Int F$ by 
\begin{equation}
\label{eq:tensoring_morphisms_in_G_construction}
(f,g) \otimes_F (f',g') = (f \otimes f', \Phi_{d,d'}(g,g')). 
\end{equation}
Since $\Phi$ is a natural transformation, the diagram
\[
\begin{tikzcd}
    Fc \times Fc'
    \arrow[r,"{\Phi_{c,c'}}"]
    \arrow[d,"Ff \times Ff'",swap]
    &
    F(c \otimes c')
    \arrow[d,"F (f \otimes f')"]
    \\
    Fd \times Fd'
    \arrow[r,"\Phi_{d,d'}",swap]
    &
    F(d \otimes d')
\end{tikzcd}\]
commutes, so $\Phi_{d,d'}(g,g')$ is a morphism from $\Phi_{d,d'}(Ff \times Ff')(x,x') = F (f \otimes f')( \Phi_{c,c'} (x,x') )$ to $\Phi_{d,d'}(y,y')$ as required.

We define the unit object of $\Int F$ to be $I_F = (I, \phi)$. Then $I_F \otimes_F (c,x) = (I, \phi) \otimes_F (c,x) = (I \otimes c, \mu_{I, c}(\phi, x))$. We take the first component of the left unitor $\lambda^F_{(c,x)} \maps (I \otimes c, \mu_{I, c}(\phi, x)) \to (c,x)$ to be the map $\lambda_c \maps I \otimes c \to c$. The second component must then be a morphism from $F\lambda_c\Phi_{I,c}(\phi,x)$ to $x$. To meet this condition we define the left unitor for $\int F$ to be
\begin{equation}
\label{eq:left_unitor}
    \lambda^F_{(c,x)} = (\lambda_c, 1_x)
\end{equation}
Since $F$ is a lax monoidal functor, the diagram
\[
\begin{tikzcd}
    I_\Cat \times Fc
    \arrow[r,"\phi \times Fc"]
    \arrow[d,"\lambda^\Cat_{I_\Cat,Fc}",swap]
    &
    F(I) \times Fc
    \arrow[d,"{\Phi_{I,c}}"]
    \\
    Fc
    &
    F(I \otimes c)
    \arrow[l,"F\lambda_c"]
\end{tikzcd}\]
commutes, giving the equation
\[
    F\lambda_c \Phi_{I,c} (\phi, x) = x
\]
as required. Similarly, we define the right unitor to be 
\begin{equation}
\label{eq:right_unitor}
\rho^F_{(c,x)} = (\rho_c, 1_x). 
\end{equation}

We take the first component of the associator $\alpha^F_{(c,x), (d,y), (e,z)}$ to be $\alpha_{c,d,e}$. The second component must then be a morphism from $F\alpha_{c,d,e} \Phi_{c\otimes d, e}(\Phi_{c,d}(x, y),z)$ to $\Phi_{c,d\otimes e}(x,\Phi_{d,e}(y,z))$. 
However, these two objects are equal, since the diagram 
\[
\begin{tikzcd}
    (Fc \times Fd) \times Fe
    \arrow[r,"\Phi_{c,d} \times Fe"]
    \arrow[d,"\alpha^\Cat_{Fc,Fd,Fe}",swap]
    &
    F(c \otimes d) \times Fe
    \arrow[d,"\Phi_{c \otimes d, e}"]
    \\
    Fc \times (Fd \times Fe)
    \arrow[d,"Fc \times \Phi_{d,e}",swap]
    &
    F( (c \otimes d) \otimes e)
    \arrow[d,"F\alpha_{c,d,e}"]
    \\
    Fc \times F(d \otimes e)
    \arrow[r,"\Phi_{c, d \otimes e}",swap]
    &
    F(c \otimes (d \otimes e) )
\end{tikzcd}\]
commutes. Thus, we can meet this condition by defining the associator for $\Int F$ to be
\begin{equation}
\label{eq:associator}
    \alpha^F_{(c,x), (d,y), (e,z)} = (\alpha_{c,d,e}, 1_{\Phi_{c,d\otimes e}(x,\Phi_{d,e}(y,z))}).
\end{equation}

\begin{thm}
\label{thm:G_construction_for_mon_cats}
    If $F \maps \C \to \Cat$ is a lax monoidal functor then $\Int F$ becomes a monoidal category when equipped with the above tensor product, unit object, unitors and associator. 
\end{thm}
\begin{proof}
    Since $C$ is a monoidal category, the following diagrams commute.
    \[\begin{tikzcd}
        (c \otimes I) \otimes d
        \arrow[dr,"\rho_c \otimes d"]
        \arrow[d,"\alpha_{c,I,d}",swap]
        \\
        c \otimes (I \otimes d)
        \arrow[r,"c \otimes \lambda_d",swap]
        &
        c \otimes d
    \end{tikzcd}\]
    \[\begin{tikzcd}
        ((b\otimes c) \otimes d) \otimes e
        \arrow[r,"\alpha_{b \otimes c,d,e}"]
        \arrow[d,"\alpha_{b,c,d} \otimes e",swap]
        &
        (b\otimes c) \otimes (d \otimes e)
        \arrow[dd,"\alpha_{b,c,d \otimes e}"]
        \\
        (b \otimes (c \otimes d)) \otimes e
        \arrow[d,"\alpha_{b,c \otimes d, e}",swap]
        \\
        b \otimes ((c\otimes d) \otimes e)
        \arrow[r,"b \otimes \alpha_{c,d,e}",swap]
        &
        b \otimes (c \otimes (d\otimes e))
    \end{tikzcd}\]
    
    It then follows that the corresponding diagrams also commute for $\Int F$, $\alpha^F$, $\lambda^F$, and $\rho^F$.
    
    \[\begin{tikzcd}
        ((c,x) \otimes_F I_F) \otimes_F (d,y)
        \arrow[dr,"\rho^F_{(c,x)} \otimes_F {(d,y)}"]
        \arrow[d,"\alpha^F_{(c,x),I_F,(d,y)}",swap]
        \\
        (c,x) \otimes_F (I_F \otimes_F (d,y))
        \arrow[r,"{(c,x)} \otimes_F \lambda^F_{(d,y)}",swap, outer sep = 2pt]
        &
        (c,x) \otimes_F (d,y)
    \end{tikzcd}\]
    \[\begin{tikzcd}
        (((b,w)\otimes_F (c,x)) \otimes_F (d,y)) \otimes_F (e,z)
        \arrow[r,"\alpha^F_{(b,w) \otimes_F (c,x),(d,y),(e,z)}", outer sep = 3pt]
        \arrow[d,"\alpha^F_{(b,w),(c,x),(d,y)} \otimes_F {(e,z)}",swap]
        &
        ((b,w)\otimes_F (c,x)) \otimes_F ((d,y) \otimes_F (e,z))
        \arrow[dd,"\alpha^F_{(b,w),(c,x),(d,y) \otimes_F (e,z)}"]
        \\
        ((b,w) \otimes_F ((c,x) \otimes_F (d,y))) \otimes_F (e,z)
        \arrow[d,"\alpha^F_{(b,w),(c,x) \otimes_F (d,y), (e,z)}",swap]
        \\
        (b,w) \otimes_F (((c,x)\otimes_F (d,y)) \otimes_F (e,z))
        \arrow[r,"{(b,w)} \otimes_F \alpha^F_{(c,x),(d,y),(e,z)}",swap, outer sep = 3pt]
        &
        (b,w) \otimes_F ((c,x) \otimes_F ((d,y)\otimes_F (e,z)))
    \end{tikzcd}\] 
\end{proof}
 
Next we show that a morphism in $\MICat$ gives a lax monoidal functor. Recall that such a morphism is a quadruple $(G,\Gamma, \gamma,\gn) \maps (F,\Phi,\phi) \to (F', \Phi',\phi')$ where
\begin{align*}
    G &\maps C \to C'\\
    \Gamma_{c,d} &\maps Gc \otimes' Gd \to G(c \otimes d) \\
    \gamma &\maps I_{\C'} \to G(I) \\
    \gn &\maps F \To F'G.
\end{align*} 
We already know how to get a functor $\Ghat \maps \Int F \to \Int F'$ from this data.  We next define $\Gahat$ and $\gahat$ to make $\Ghat$ into a lax monoidal functor:
\begin{equation}
\label{eq:lax_tensor}
    \Gahat_{(c,x),(d,x')} = (\Gamma_{c,d},1),
\end{equation}
\begin{equation}
\label{eq:lax_unit}
    \gahat = (\gamma, 1).
\end{equation}
One can check that these have the required source and target.

\begin{thm}
\label{thm:G_construction_functorial_for_mon_cats}
    There exists a unique functor, the \define{monoidal Grothendieck \hfill \break construction}
    \[\textstyle{\Int} \maps \MICat \; \to \; \MC,\]
    that sends any object $F$ to the monoidal category $\Int F$ given in Thm.\ \ref{thm:G_construction_for_mon_cats} and sends any morphism $G \maps F \to F'$ to the lax monoidal functor $\Ghat \maps \Int F \to \Int F'$ defined above. 
\end{thm}
\begin{proof}
    Uniqueness follows because the theorem specifies $\int$ on objects and morphisms.
    For existence, we need to check that $\Ghat$ is a lax monoidal functor and that $\int$ preserves composition and identities.
    
    We already know that $\Ghat $ is a functor. We start by checking that $\Gahat$ is a natural transformation. Let $f \maps c\to d$, $f' \maps c'\to d'$ be morphisms in $\C$. Since $\Gamma$ is a natural transformation, the following diagram commutes
    \[\begin{tikzcd}
        Gc \otimes' Gc'
        \arrow[r,"\Gamma_{c,c'}"]
        \arrow[d,"Gf \otimes' Gf'",swap]
        &
        G(c \otimes c')
        \arrow[d,"G(f \otimes f')"]
        \\
        Gd \otimes' Gd'
        \arrow[r,"\Gamma_{d,d'}",swap]
        &
        G(d \otimes d')
    \end{tikzcd}\]
    giving the equation $G(f\otimes f') \circ \Gamma_{c,c'} = \Gamma_{d,d'} \circ (Gf \otimes' Gf')$. Since $\gn$ is a monoidal natural transformation, the following diagram commutes 
    \[\begin{tikzcd}
        Fd \times Fd'
        \arrow[r,"\gn_d \times \gn_{d'}"]
        \arrow[d,"\Phi_{d,d'}",swap]
        &
        F'Gd \times F'Gd'
        \arrow[d,"F'\Gamma_{d,d'}\Phi'_{Gd,Gd'}"]
        \\
        F(d \otimes d')
        \arrow[r,"\gn_{d\otimes d'}",swap]
        &
        F'G(d \otimes d')
    \end{tikzcd}\]
    giving the equation $\gn_{d\otimes d'} \Phi_{d,d'} = F'\Gamma_{d,d'}\Phi'_{Gd,Gd'} (\gn_d \times \gn_{d'})$. Then
    \begin{align*}
        \Ghat ((f,g) \otimes_F (f',g')) \circ \Gahat_{(c,x),(c',x')}
        &= \Ghat (f \otimes f', \Phi_{d,d'}(g,g')) \circ (\Gamma_{c,c'}, 1)
        \\&= (G(f \otimes f'), \gn_{d\otimes d'} \Phi_{d,d'}(g,g')) \circ (\Gamma_{c,c'}, 1)
        \\&= (G(f \otimes f') \circ \Gamma_{c,c'}, \gn_{d\otimes d'} \Phi_{d,d'}(g,g'))
        \\&= (\Gamma_{d,d'} \circ (Gf \otimes' Gf'), F'\Gamma_{d,d'}\Phi'_{Gd,Gd'} (\gn_dg,\gn_{d'}g'))
        \\&= (\Gamma_{d,d'},1) \circ (Gf\otimes Gf', \Phi'_{Gd,Gd'}(\gn_dg,\gn_{d'}g')
        \\&= (\Gamma_{d,d'},1) \circ (Gf,\gn_dg) \otimes_{F'} (Gf',\gn_{d'}g')
        \\&= \Gahat_{(d,y),(d',y')} \circ \Ghat(f,g) \otimes_{F'} \Ghat (f',g')
    \end{align*}
    which tells us that the following diagram commutes.
    \[\begin{tikzcd}
        \Ghat  (c,x) \otimes_{F'} \Ghat  (c',x')
        \arrow[r,"\Gahat_{(c,x),(c',x')}", outer sep = 4pt]
        \arrow[d,"{\Ghat  (f,g) \otimes_{F'} \Ghat (f',g')}",swap]
        &
        \Ghat  ((c,x) \otimes_F (c',x'))
        \arrow[d,"{\Ghat ((f,g) \otimes_F (f',g'))}"]
        \\
        \Ghat  (d,y) \otimes_{F'} \Ghat  (d',y')
        \arrow[r,"\Gahat_{(d,y),(d',y')}",swap, outer sep = 4pt]
        &
        \Ghat  ((d,y) \otimes_F (d',y'))
    \end{tikzcd}\]
    
    Next, we check that $\Gahat$ satisfies the necessary conditions to be a lax structure map. Since $(G, \Gamma)$ is a lax monoidal functor, the following diagrams commute.
    \[\begin{tikzcd}
        (Gc \otimes' Gd) \otimes' Ge
        \arrow[r,"\alpha'_{Gc,Gd,Ge}"]
        \arrow[d,"\Gamma_{c,d} \otimes' Ge",swap]
        &
        Gc \otimes' (Gd \otimes' Ge)
        \arrow[d,"Gc \otimes' \Gamma_{d,e}"]
        \\
        G(c \otimes d) \otimes' Ge
        \arrow[d,"\Gamma_{c \otimes d, e}",swap]
        &
        Gc \otimes' G(d \otimes e)
        \arrow[d,"\Gamma_{c, d \otimes e}"]
        \\
        G((c \otimes d) \otimes e)
        \arrow[r,"G\alpha_{c,d,e}",swap]
        &
        G(c \otimes (d \otimes e))
    \end{tikzcd}\] 
    \[\begin{tikzcd}
        I' \otimes' Gc
        \arrow[r,"\gamma \otimes' Gc"]
        \arrow[d,"\lambda'_{Gc}",swap]
        &
        GI \otimes' Gc
        \arrow[d,"\Gamma_{I,c}"]
        \\
        Gc
        &
        G(I \otimes c)
        \arrow[l,"G\lambda_c"]
    \end{tikzcd}\] 
    \[\begin{tikzcd}
        Gc \otimes' I'
        \arrow[r,"Gc \otimes' \gamma"]
        \arrow[d,"\rho'_{Gc}",swap]
        &
        Gc \otimes' GI
        \arrow[d,"\Gamma_{c,I}"]
        \\
        Gc
        &
        G(c \otimes I)
        \arrow[l,"G\rho_c"]
    \end{tikzcd}\] 
    It then follows that the corresponding diagrams also commute for $\Ghat$, $\Gahat$, and $\gahat$.
    \[\begin{tikzcd}
        (\Ghat (c,x) \otimes_{F'} \Ghat (d,y)) \otimes_{F'} \Ghat (e,z)
        \arrow[r,"\alpha^{F'}_{\Ghat (c,x),\Ghat (d,y),\Ghat (e,z)}", outer sep = 5pt]
        \arrow[d,"\Gahat_{(c,x),(d,y)} \otimes_{F'} \Ghat {(e,z)}",swap]
        &
        \Ghat (c,x) \otimes_{F'} (\Ghat (d,y) \otimes_{F'} \Ghat (e,z))
        \arrow[d,"\Ghat {(c,x)} \otimes_{F'} \Gahat_{(d,y),(e,z)}"]
        \\
        \Ghat((c,x) \otimes (d,y)) \otimes_{F'} \Ghat (e,z)
        \arrow[d,"\Gahat_{(c,x) \otimes_F (d,y), (e,z)}",swap]
        &
        \Ghat (c,x) \otimes_{F'} \Ghat((d,y) \otimes_F (e,z))
        \arrow[d,"\Gahat_{(c,x), (d,y) \otimes_F (e,z)}"]
        \\
        \Ghat(((c,x) \otimes_F (d,y)) \otimes_F (e,z))
        \arrow[r,"\Ghat\alpha^F_{(c,x),(d,y),(e,z)}",swap, outer sep = 4pt]
        &
        \Ghat((c,x) \otimes_F ((d,y) \otimes_F (e,z)))
    \end{tikzcd}\] 
    \[\begin{tikzcd}
        I_{F'} \otimes_{F'} \Ghat (c,x)
        \arrow[r,"\gahat \otimes_{F'} \Ghat {(c,x)}", outer sep = 5pt]
        \arrow[d,"\lambda^{F'}_{\Ghat (c,x)}",swap]
        &
        \Ghat I_F \otimes_{F'} \Ghat (c,x)
        \arrow[d,"\Gahat_{I,(c,x)}"]
        \\
        \Ghat (c,x)
        &
        \Ghat(I_F \otimes_F (c,x))
        \arrow[l,"\Ghat\lambda^F_{(c,x)}"]
    \end{tikzcd}\] 
    \[\begin{tikzcd}
        \Ghat (c,x) \otimes_{F'} I_{F'}
        \arrow[r,"\Ghat {(c,x)} \otimes_{F'} \gahat", outer sep = 5pt]
        \arrow[d,"\rho^{F'}_{\Ghat(c,x)}",swap]
        &
        \Ghat (c,x) \otimes_{F'} \Ghat I
        \arrow[d,"\Gahat_{(c,x),I}"]
        \\
        \Ghat (c,x)
        &
        \Ghat ((c,x) \otimes_F I)
        \arrow[l,"\Ghat\rho^F_{(c,x)}"]
    \end{tikzcd}\] 

    Finally, we check that composition is preserved.
    \begin{align*}
        ( (G,\Gamma,\gamma,\gn) \circ (G',\Gamma',\gamma',\gn') ) ^{\widehat{}}
        &= (G \circ G', G\Gamma' \circ \Gamma_{G'}, G\gamma' \circ \gamma,\gn_G \circ \gn') ^{\widehat{}}
        \\&= (\widehat{G\circ G'}, (G\Gamma' \circ \Gamma_{G'},1), (G\gamma' \circ \gamma,1))
        \\&= (\Ghat \circ \Ghat', (G\Gamma',1) \circ (\Gamma_{G'},1), (G\gamma',1) \circ (\gamma,1))
        \\&= (\Ghat \circ \Ghat', \Ghat(\Gamma',1) \circ (\Gamma_{G'},1), \Ghat(\gamma',1) \circ (\gamma,1))
        \\&= (\Ghat \circ \Ghat', \Ghat\Gahat' \circ \Gahat_{\Ghat'}, \Ghat \gahat' \circ \gahat)
        \\&= (\Ghat, \Gahat, \gahat) \circ (\Ghat', \Gahat', \gahat')
        \\&= (G,\Gamma,\gamma,\gn)^{\widehat{}} \circ (G',\Gamma',\gamma',\gn')^{\widehat{}}
        \qedhere
    \end{align*}
\end{proof}

Under some conditions the Grothendieck construction gives \emph{strict} monoidal categories and functors:

\begin{prop}
\label{prop:G_construction_for_strict_mon_cats}
    If $C$ is a strict monoidal category and $F \maps \C \to \Cat$ is a lax monoidal functor, then $\Int F$ as defined in Thm.\ \ref{thm:G_construction_for_mon_cats} is a strict monoidal category. 
\end{prop}

\begin{proof}
    This follows from Eq.\ \ref{eq:left_unitor} for the left unitor, Eq.\ \ref{eq:right_unitor} for the right unitor, and Eq.\ \ref{eq:associator} for the associator in $\Int F$. These isomorphisms are all built from the corresponding isomorphisms in $\C$ in such a way that if $\C$ is strict monoidal, so is $\Int F$.
\end{proof}

\begin{prop}
\label{prop:G_construction_for_strict_mon_functors}
    If
    \[\begin{tikzcd}
        \C
        \arrow[dr, "F"]
        \arrow[dr, ""{name=F}, swap]
        \arrow[dd, "G", swap]
        \\&
        \Cat
        \\
        \C'
        \arrow[ur, "F'", swap]
        \arrow[ur, ""{name=F'}, pos=0.43]
        \arrow[Rightarrow, from = F, to = F', "\gn", swap]
    \end{tikzcd}\]
    is a morphism in $\MICat$ such that $G$ is a strict monoidal functor, then $\Ghat \maps \Int F \to \Int F'$ as defined in Thm.\ \ref{thm:G_construction_functorial_for_mon_cats} is a strict monoidal functor. 
\end{prop}

\begin{proof}
    This follows from Eq.\ \ref{eq:lax_tensor} and Eq.\ \ref{eq:lax_unit}, which give the morphisms describing how $\Ghat$ laxly preserves of the tensor product and unit for the tensor product. These morphisms are built from the corresponding morphisms for $G$ in such a way that if $G$ is strict monoidal, so is $\Ghat$.
\end{proof}

\subsection{The braided Grothendieck construction}

Next we consider the braided case.

\begin{defn}
\label{defn:BMC}
    Let $\BMC$ be the category with small braided monoidal categories as objects and lax braided monoidal functors as morphisms. 
\end{defn}

\begin{defn} 
\label{defn:BMICat}
Let $\BMICat$ be the category where:
\begin{itemize}
    \item objects are pairs $(\C,G)$ where $\C$ is a small braided monoidal category and $G \maps \C \to \Cat$ is a lax braided monoidal functor.
    \item morphisms from $(\C_1, G_1)$ to $(\C_2,G_2)$ are pairs $(G,\gn)$ where $G \maps \C_1 \to \C_2$ is a lax braided monoidal functor and $\gn \maps F \To F' \circ G $ is a braided monoidal natural transformation:
    \[\begin{tikzcd}
    \C
    \arrow[dr, "F"]
    \arrow[dr, ""{name=F}, swap]
    \arrow[dd, "G", swap]
    \\&
    \Cat
    \\
    \C'
    \arrow[ur, "F'", swap]
    \arrow[ur, ""{name=F'}, pos=0.43]
    \arrow[Rightarrow, from = F, to = F', "\gn", swap]
\end{tikzcd}\]
    \end{itemize}
\end{defn}

Let $\C$ be a braided monoidal category with  braiding $B_{c,d} \maps c \otimes d \to d \otimes c$. Let $F \maps \C \to \Cat$ a lax braided monoidal functor with lax structure map $\Phi$, so that the following diagram commutes:
\[
\begin{tikzcd}
    Fc\times Fd
    \arrow[r,"\Phi_{c,d}"]
    \arrow[d,"B_{c,d}",swap]
    &
    F(c \otimes d)
    \arrow[d,"FB_{c,d}"]
    \\
    Fd \times Fc
    \arrow[r,"\Phi_{d,c}",swap]
    &
    F(d \otimes c).
\end{tikzcd}\]
We claim that in this situation we can make $\Int F$ into a braided monoidal category, giving it a braiding 
\[B^F_{(c,x),(d,y)} \maps (c \otimes d, \Phi_{c,d}(x,y)) \to (d \otimes c, \Phi_{d,c}(y,x)) .\]
We take the first component of this morphism to be $B_{c,d}$.  The second component must then be a morphism from $FB_{c,d}(\Phi_{c,d}(x,y))$ to $\Phi_{d,c}(y,x)$, but 
\begin{align*} 
    FB_{c,d}(\Phi_{c,d}(x,y)) 
    &= \Phi_{d,c}(B_{c,d}(x,y)) 
    \\&= \Phi_{d,c}(y,x).
\end{align*}
so if we define the braiding in $\int F$ by
\[B^F_{(c,x),(d,y)} = (B_{c,d}, 1)\]
then this condition is met.

\begin{thm}
\label{thm:G_construction_for_braided_cats}
    If $F \maps \C \to \Cat$ is a lax braided monoidal functor, then $\Int F$ made monoidal as in Thm.\ \ref{thm:G_construction_for_mon_cats} and given the braiding $B^F$ is a braided monoidal category.
\end{thm}
\begin{proof}
    We need to show that $B^F$ is a natural transformation and that it obeys the hexagon identities. Let $(f,g) \maps (c,x) \to (d,y)$ and $(f',g') \maps (c',x') \to (d',y')$ be morphisms in $C$. Since $\C$ is braided monoidal with the braiding $B$, the following diagram commutes:
    \[\begin{tikzcd}
        c \otimes c'
        \arrow[r,"B_{c,c'}"]
        \arrow[d,"f \otimes f'",swap]
        &
        c' \otimes c
        \arrow[d,"f' \otimes f"]
        \\
        d \otimes d'
        \arrow[r,"B_{d,d'}",swap]
        &
        d' \otimes d
    \end{tikzcd}\]
    giving the equation $(f'\otimes f) \circ B_{c,c'} = B_{d,d'} \circ (f \otimes f')$. Since $F$ is lax braided monoidal, the following diagram commutes:
    \[\begin{tikzcd}
        Fd \times Fd'
        \arrow[r,"B_{Fd,Fd'}"]
        \arrow[d,"\Phi_{d,d'}",swap]
        &
        Fd' \times Fd
        \arrow[d,"\Phi_{d',d}"]
        \\
        F(d \otimes d')
        \arrow[r,"FB_{d,d'}",swap]
        &
        F(d' \otimes d')
    \end{tikzcd}\]
    giving the equation $FB_{d,d'} (\Phi_{d,d'}) = \Phi_{d',d}(B_{Fd,Fd'})$.
    Thus
    \begin{align*}
        B^F_{(d,y),(d',y')} \circ ((f,g) \otimes_F (f',g'))
        &= (B_{d,d'}, 1) \circ (f \otimes f', \Phi_{d,d'}(g,g'))
        \\&= (B_{d,d'} \circ f\otimes f', FB_{d,d'}(\Phi_{d,d'}(g,g')))
        \\&= (f'\otimes f \circ B_{c,c'}, \Phi_{d',d}(B_{Fd,Fd'}(g,g') )
        \\&= (f' \otimes f \circ B_{c,c'}, \Phi_{d',d}(g',g))
        \\&= (f' \otimes f, \Phi_{d',d}(g',g)) \circ (B_{c,c'}, 1)
        \\&= ((f',g') \otimes_F (f,g)) \circ B^F_{(c,x),(c',x')}.
    \end{align*}
    This tells us that the following diagram commutes, and thus $B^F$ is natural. 
    \[\begin{tikzcd}
        (c,x) \otimes_F (c',x') 
        \arrow[r,"B^F_{(c,x),(c',x')}", outer sep = 3pt]
        \arrow[d,"{(f,g) \otimes_F (f',g')}",swap]
        &
        (c',x') \otimes_F (c,x)
        \arrow[d,"{(f',g') \otimes_F (f,g)}"]
        \\
        (d,y) \otimes_F (d',y') 
        \arrow[r,"B^F_{(d,y),(d',y')}",swap, outer sep = 2pt]
        &
        (d',y') \otimes_F (d,y)
    \end{tikzcd}\]
    
    Next we show that the necessary diagrams commute to make $B^F$ a braiding for $\Int F$. 
    Notice that the diagrams
    \[\begin{tikzcd}
        c \otimes ( d \otimes e )
        \arrow[r, "\alpha\inv_{c,d,e}"]
        \arrow[d, "B_{c, d \otimes e}",swap]
        &
        (c \otimes d) \otimes e
        \arrow[d, "B_{c,d} \otimes e"]
        \\
        (d \otimes e ) \otimes c
        &
        (d \otimes c) \otimes e
        \arrow[d, "\alpha_{d, c, e}"]
        \\
        d \otimes (e  \otimes c )
        \arrow[u, "\alpha\inv_{d,e,c}"]
        &
        d \otimes (c \otimes e)
        \arrow[l, "d \otimes B_{c, e}"]
    \end{tikzcd}\]
    \[\begin{tikzcd}
        (c \otimes d) \otimes e
        \arrow[r, "\alpha_{c,d,e}"]
        \arrow[d, "B_{c \otimes d, e}",swap]
        &
        c \otimes (d \otimes e)
        \arrow[d, "c \otimes B_{d,e}"]
        \\
        e \otimes (c \otimes d)
        &
        c \otimes (e \otimes d)
        \arrow[d, "\alpha\inv_{c,e,d}"]
        \\
        (e \otimes c) \otimes d
        \arrow[u, "\alpha_{e,c,d}"]
        &
        (c \otimes e) \otimes d
        \arrow[l, "B_{c,e} \otimes d"]
    \end{tikzcd}\]
    commute.
    Then the corresponding diagrams for $\alpha^F$ and $B^F$ commute.
    \[\begin{tikzcd}
        (c,x) \otimes_F ( (d,y) \otimes_F (e,z) )
        \arrow[r, "(\alpha^F)\inv_{(c,x),(d,y),(e,z)}", outer sep = 3pt]
        \arrow[d, "B^F_{(c,x), (d,y) \otimes_F (e,z)}",swap]
        &
        ((c,x) \otimes_F (d,y)) \otimes_F (e,z)
        \arrow[d, "B^F_{(c,x),(d,y)} \otimes {(e,z)}"]
        \\
        ((d,y) \otimes_F (e,z) ) \otimes_F (c,x)
        &
        ((d,y) \otimes_F (c,x)) \otimes_F (e,z)
        \arrow[d, "\alpha^F_{(d,y), (c,x), (e,z)}"]
        \\
        (d,y) \otimes_F ((e,z)  \otimes_F (c,x) )
        \arrow[u, "(\alpha^F)\inv_{(d,y),(e,z),(c,x)}"]
        &
        (d,y) \otimes_F ((c,x) \otimes_F (e,z))
        \arrow[l, "{(d,y)} \otimes_F B^F_{(c,x), (e,z)}", outer sep = 2pt]
    \end{tikzcd}\]
    \[\begin{tikzcd}
        ((c,x) \otimes_F (d,y)) \otimes_F (e,z)
        \arrow[r, "\alpha^F_{(c,x),(d,y),(e,z)}", outer sep = 3pt]
        \arrow[d, "B^F_{(c,x) \otimes_F (d,y), (e,z)}",swap]
        &
        (c,x) \otimes_F ((d,y) \otimes_F (e,z))
        \arrow[d, "{(c,x)} \otimes_F B^F_{(d,y),(e,z)}"]
        \\
        (e,z) \otimes_F ((c,x) \otimes_F (d,y))
        &
        (c,x) \otimes_F ((e,z) \otimes_F (d,y))
        \arrow[d, "(\alpha^F)\inv_{(c,x),(e,z),(d,y)}"]
        \\
        ((e,z) \otimes_F (c,x)) \otimes_F (d,y)
        \arrow[u, "\alpha^F_{(e,z),(c,x),(d,y)}"]
        &
        ((c,x) \otimes_F (e,z)) \otimes_F (d,y)
        \arrow[l, "B^F_{(c,x),(e,z)} \otimes_F {(d,y)}", outer sep = 4pt]
    \end{tikzcd}\]
\end{proof}

\begin{thm}
\label{thm:G_construction_functorial_for_braided_cats}
    There exists a unique functor, the \define{braided Grothendieck \hfill \break construction} 
    \[\textstyle{\Int} \maps \BMICat \to \BMC, \] 
    that sends any object $F$ to the braided monoidal category $\Int F$ given in Thm.\ \ref{thm:G_construction_for_braided_cats} and sends any morphism $G \maps F \to F'$ to the lax braided monoidal functor $\Ghat \maps \Int F \to \Int F'$ defined above.
\end{thm}
\begin{proof}
     Uniqueness follows because the theorem specifies $\int$ on objects and morphisms.   From Thm.\ \ref{thm:G_construction_functorial_for_mon_cats}   we already know that $\Ghat$ is lax monoidal for any morphism in $\MICat$ and that $\int$ preserves composition and identities. Thus, for existence all we need to show is that $\Ghat$ is in fact braided.
    
    Since $G$ is braided monoidal, the following diagram commutes:
    \[\begin{tikzcd}
        Gc \otimes' Gd
        \arrow[r,"B'_{Gc,Gd}"]
        \arrow[d,"\Gamma_{c,d}",swap]
        &
        Gd \otimes' Gc
        \arrow[d,"\Gamma_{d,c}"]
        \\
        G(c \otimes d)
        \arrow[r,"GB_{c,d}",swap]
        &
        G(d \otimes c)
    \end{tikzcd}\]
    Thus, the corresponding diagram commutes:
    \[\begin{tikzcd}
        \Ghat (c,x) \otimes_{F'} \Ghat (d,y)
        \arrow[r,"B^{F'}_{\Ghat (c,x),\Ghat (d,y)}", outer sep = 6pt]
        \arrow[d,"\Gahat_{(c,x),(d,y)}",swap]
        &
        \Ghat (d,y) \otimes_{F'} \Ghat (c,x)
        \arrow[d,"\Gahat_{(d,y),(c,x)}"]
        \\
        \Ghat ((c,x) \otimes_F (d,y))
        \arrow[r,"\Ghat B^F_{(c,x),(d,y)}",swap, outer sep = 3pt]
        &
        \Ghat ((d,y) \otimes_F (c,x))
    \end{tikzcd}\]
    This shows that $\Ghat$ is braided.
\end{proof}

\subsection{The symmetric Grothendieck construction}

Finally we turn to the symmetric monoidal case. 

\begin{defn}
\label{defn:SMC}
Let $\SMC$ be the category with small symmetric monoidal categories as objects and lax symmetric monoidal functors as morphisms. 
\end{defn}

\noindent
We defined the category $\SMICat$ in Def.\ \ref{defn:SMICat}. So, we are ready to state our main result:

\begin{thm}
\label{thm:G_construction_functorial_for_symmetric_cats}
    There exists a unique functor, the \define{symmetric Grothendieck construction}
    \[
        \Int \maps \SMICat \; \to \; \SMC 
    \] 
    that acts on objects and morphisms as in Thm.\
    \ref{thm:G_construction_functorial_for_braided_cats}.
\end{thm}
\begin{proof}
    In Thm.\ 
    \ref{thm:G_construction_functorial_for_braided_cats} we obtained a functor $\Int \maps \BMICat \to \BMC$, so for both existence and uniqueness we need only check that if an object $F \maps \C \to \Cat$ of $\BMICat$ has $C$ symmetric then $\Int F$ is symmetric as well. This is straightforward from the formula for the braiding in $\Int F$. 
\end{proof}

We conclude with a strict version of the above result, which we use in the next section.

\begin{defn}
\label{defn:SSMC}
    Let $\SSMC$ be the category with small strict symmetric monoidal categories as objects and strict symmetric monoidal functors as morphisms. 
\end{defn}

\begin{defn} 
\label{defn:SSMICat}
Let $\SSMICat$ be the category where:
\begin{itemize}
    \item objects are pairs $(\C,G)$ where $\C$ is a small strict symmetric monoidal category and $G \maps \C \to \Cat$ is a lax symmetric monoidal functor.
    \item morphisms from $(\C_1, G_1)$ to $(\C_2,G_2)$ are pairs $(G,\gn)$ where $G \maps \C_1 \to \C_2$ is a strict symmetric monoidal functor and $\gn \maps F \To F' \circ G $ is a symmetric monoidal natural transformation:
    \[\begin{tikzcd}
    \C
    \arrow[dr, "F"]
    \arrow[dr, ""{name=F}, swap]
    \arrow[dd, "G", swap]
    \\&
    \Cat
    \\
    \C'
    \arrow[ur, "F'", swap]
    \arrow[ur, ""{name=F'}, pos=0.43]
    \arrow[Rightarrow, from = F, to = F', "\gn", swap]
\end{tikzcd}\]
\end{itemize}
\end{defn}

\begin{thm}
\label{thm:G_construction_functorial_for_strict_symmetric_cats}
 There is a unique functor, the \define{strict
 symmetric Grothendieck construction}
\[\textstyle{\Int} \maps \SSMICat \; \to \; \SSMC ,\] 
that acts on objects and morphisms as in Thm.\
\ref{thm:G_construction_functorial_for_symmetric_cats}.
\end{thm}

\begin{proof}
To prove this it suffices to check these claims:
\begin{enumerate}
    \item If $F \maps \C \to \Cat$ is an object of $\SMICat$ with $\C$ a \emph{strict} symmetric monoidal category, then $\Int F$ is a strict symmetric monoidal category.
    \item If
     \[\begin{tikzcd}
    \C
    \arrow[dr, "F"]
    \arrow[dr, ""{name=F}, swap]
    \arrow[dd, "G", swap]
    \\&
    \Cat
    \\
    \C'
    \arrow[ur, "F'", swap]
    \arrow[ur, ""{name=F'}, pos=0.43]
    \arrow[Rightarrow, from = F, to = F', "\gn", swap]
\end{tikzcd}\]
    is a morphism in $\SMICat$ with $G$ a \emph{strict} symmetric monoidal functor, then $\Ghat \maps \Int F \to \Int F$ is a strict symmetric monoidal functor. 
\end{enumerate}
The first of these follows from Prop.\ \ref{prop:G_construction_for_strict_mon_cats}, while the second follows from Prop.\ \ref{prop:G_construction_for_strict_mon_functors}. \end{proof}

\subsection{A 2-categorical perspective}
\label{subsec:mon_fib}

We have explained how to use the Grothendieck construction to build a symmetric monoidal category $\Int F$ from a lax symmetric monoidal functor $F \maps \C \to \Cat$. In what follows, all we really need are the explicit formulas for how this works. Still, it seems worthwhile to set this result in a larger context. This requires a 2-categorical perspective on fibrations and indexed categories. A full review of the prerequisites would be rather lengthy, so we only recall a few facts. The main ideas go back to Grothendieck \cite{Grothendieck}, but they have been developed and explained by many subsequent authors \cite{Borceux,Hermida, Jacobs,Johnstone,Vasilakopoulou}, whose works can be consulted for details. To conform to this literature we now replace $\C$ with $\C^{\op}$, which lets us work with fibrations rather than opfibrations. This makes no real difference for the categories $\C$ we are mainly interested in, which are groupoids.

A functor $F \maps \C^{\op} \to \Cat$ is sometimes called a `split indexed category'. The Grothendieck construction builds from this a category $\Int F$. The objects and morphism of $\Int F$ are pairs whose first component belongs to $\C$. Projecting onto this first component gives a functor
\[p \maps \Int F \to \C  .\]
This functor is equipped with some extra structure that makes it into a `split fibration'. The idea is that an object of $\Int F$ is an object of $\C$ equipped with extra structure, and the splitting gives a well-behaved way to take this extra structure and pull it back along any morphism in $\C$. For example, given a bijection of finite sets, and a simple graph on the domain of this bijection, we can pull it back to the codomain. 

Conversely, from a split fibration we can recover a split indexed category. Indeed, split fibrations are essentially `the same' as split indexed categories. To express this fact clearly, we need a couple of 2-categories that are nicely explained by Jacobs \cite[Sec.\ 1.7]{Jacobs}. First, we need the 2-category $\ICat$ of split indexed categories, where an object is a functor $F \maps \C^{\op} \to \Cat$ for an arbitrary small category $\C$, a morphism is a pseudonatural transformation, and a 2-morphism is a modification. Second, we need the 2-category $\Fib_{\spl}$ where an object is a split fibration, a  morphism is a morphism of fibrations that preserves the splitting, and a 2-morphism is a fibered natural transformation. The map sending any functor $F \maps \C \to \Cat$ to the split fibration $p \maps \Int F \to \C$ then extends to an equivalence of 2-categories 
\[G \maps \ICat \stackrel{\sim}{\longrightarrow} \Fib_{\spl} .\]
There is also a 2-functor 
\[\mathrm{dom} \maps \Fib_{\spl} \longrightarrow \Cat \] that sends a split fibration $F \maps \D \to \C$ to its domain category $\D$, and the composite 
\[\Int = \mathrm{dom} \circ G \maps \ICat \longrightarrow \Cat \]
is none other than the Grothendieck construction, now regarded as a 2-functor rather than a mere functor. For details see Jacobs \cite[Thm.\ 1.10.7]{Jacobs}.

A `pseudomonoid' is a generalization of a monoidal category which makes sense in any 2-category with finite products (or even more generally). We can also define braided and symmetric pseudomonoids, which generalize braided and symmetric monoidal categories \cite{DS}. For any 2-category $\twocat$ with finite products, let:
\begin{itemize}
    \item $\m\twocat$ be the 2-category of pseudomonoids in $\twocat$, monoidal morphisms, and monoidal transformations, 
    \item $\bm\twocat$ be the 2-category of braided pseudomonoids in $\twocat$, braided monoidal morphisms, and braided monoidal transformations, 
    \item $\sm\twocat$ be the 2-category of symmetric pseudomonoids in $\twocat$, symmetric monoidal morphisms, and symmetric monoidal transformations.
\end{itemize}
All these concepts are defined by McCrudden \cite{McCrudden}.  The notation here is compatible with that introduced earlier in this section, except that now we are working 2-categorically. That is, $\MC$, $\BMC$ and $\SMC$ are 2-categories whose underlying categories were already defined in Defs.\ \ref{defn:MC}, \ref{defn:BMC} and \ref{defn:SMC}. More interesting is that the 2-category $\Fib_{\spl}$ has finite products, so the equivalent 2-category $\ICat$ does as well, and one can prove that the 2-categories $\m\ICat, \bm\ICat$ and $\sm\ICat$ are 2-categories whose underlying categories match those defined in Defs.\ \ref{defn:MICat}, \ref{defn:BMICat} and \ref{defn:SMICat}, at least after replacing $\C$ with $\C^{\op}$.  We leave the verification of this as an exercise for the reader.

The equivalence $G \maps \ICat \to \Fib_{\spl}$ preserves products, and so does the 2-functor $\mathrm{dom} \maps \Fib_{\spl} \to \Cat$. Thus, so does the composite 2-functor
$\Int = \mathrm{dom} \circ G$. It thus induces 2-functors
\[
\begin{array}{lcc}
    \int \maps \m\ICat &\to& \MC\\ \\
    \int \maps \bm\ICat &\to& \BMC\\ \\
    \int \maps \sm\ICat &\to & \SMC.
\end{array}
\]
These 2-functors match those given in Thms.\ \ref{thm:G_construction_functorial_for_mon_cats}, \ref{thm:G_construction_functorial_for_braided_cats}, and \ref{thm:G_construction_functorial_for_symmetric_cats}, at least after replacing $\C$ with $\C^{\op}$.

The reader conversant with fibrations may wonder why we are restricting attention to \emph{split} indexed categories and \emph{split} fibrations. Only the split case seems relevant to network models. However, everything we have just done also works for general indexed categories and fibrations.  This is shown by Vasilakopolou and the second author in a paper that develops the 2-categorical approach sketched here \cite{MV}.

\section{Operads from network models}
\label{sec:operads}

Next we describe the operad associated to a network model. There is a standard method of constructing an untyped operad from an object $x$ in a strict symmetric monoidal category $\C$. Namely, we define the set of $n$-ary operations to be $\hom_{\C}(x^{\otimes n},x)$, and compose these operations using composition in $\C$. This gives the so-called {\bf endomorphism operad} of $x$. Here we use the generalization of this idea to the typed case, using \emph{all} the objects of $\C$ as the types of the operad. 

We assume familiarity with typed operads: these are often called `colored' operads, with the types called `colors' \cite{Yau}.  In what follows we let $\Ob(\C)$ be the set of objects of a small category $\C$.

\begin{prop}
\label{prop:operad_from_symmoncat}
    If $\C$ is a small strict symmetric monoidal category then there is an $\Ob(\C)$-typed operad $\op(\C)$ for which:
    \begin{itemize}
        \item the set of operations $\op(\C)(c_1, \dots, c_k; c)$ is defined to be $\hom_\C(c_1 \otimes \dots \otimes c_k, c)$,
        \item given operations   
        \[f\in \hom_\C(c_1 \otimes \cdots \otimes c_k; c) \]
        and 
        \[g_i \in \hom_\C(c_{ij_1} \otimes \cdots \otimes c_{ij_i},c_i) 
        \]
        for $1 \le i \le k$, their composite is defined by
        \begin{equation}
        \label{eq:composition_of_operations}
        f \circ (g_1, \dots, g_k) = f \circ (g_1 \otimes \cdots \otimes g_k) .
        \end{equation}
        \item identity operations are identity morphisms in $\C$, and
        \item the action of $S_k$ on $k$-ary operations is defined using the braiding in $\C$.
    \end{itemize}
\end{prop}

\begin{proof}
    The various axioms of a colored operad can be checked for $\op(\C)$ using the corresponding laws in the definition of a strict symmetric monoidal category.  The associativity axiom for $\op(\C)$ follows from associativity of composition and the functoriality of the tensor product in $\C$.  The left and right unit axioms for $\op(\C)$ follow from the unit laws for composition and the functoriality of the tensor product in $\C$.  The two equivariance axioms for $\op(\C)$ follow from the laws governing the braiding in $\C$. 
\end{proof}

Given a network model $F \maps \S(C) \to \Cat$, we can use the strict symmetric Grothendieck construction of Thm.\ \ref{thm:G_construction_functorial_for_strict_symmetric_cats} to define a strict symmetric monoidal category $\Int F$. We can then use Prop.\ \ref{prop:operad_from_symmoncat} to build an operad $\op(\Int F)$.

\begin{defn}
\label{defn:CN}
    Given a network model $F \maps \S(C) \to \Cat$, define the \define{network operad} $\CN_F$ to be $\op(\Int F)$. 
\end{defn}

If $F \maps \S(C) \to \Mon$, the objects of $\Int F$ correspond to objects of $\S(C)$, which are formal expressions of the form
\[c_1 \otimes \cdots \otimes c_k \]
with $k \in \N$ and $c_i \in C$. Thus, the network operad $\CN_F$ is a typed operad where the types are expressions of this form: that is, ordered $k$-tuples of elements of $C$.

Now suppose that $F$ is a one-colored network model, so that $F \maps \S \to \Mon$. Then the objects of $\S$ are simply natural numbers, so $\CN_F$ is an $\N$-typed operad. Given $n_1, \dots, n_k, n \in \N$, we have
\[\CN_F(n_1, \dots, n_k; n) = \hom_{\Int \! F}(n_1 + \cdots + n_k, n). \]
By the definition of the Grothendieck construction, a morphism in this homset is a pair consisting of a bijection $\sigma \maps n_1 + \cdots + n_k \to n$ and an element of the monoid $F(n)$. So, we have
\begin{equation}
\label{eq:operations_in_CN}
 \CN_F(n_1, \dots, n_k; n) = \left\{ 
\begin{array}{cl}  S_n \times F(n) & \textrm{if } n_1 + \cdots n_k = n \\
\emptyset & \textrm{otherwise.} \\
\end{array} \right. 
\end{equation}

Here is the basic example:

\begin{ex}[\textbf{Simple network operad}]
\label{ex:simple_network_operad} 
    If $\SG \maps \S \to \Mon$ is the network model of simple graphs in Ex.\ \ref{ex:simple_graph}, we call $\CN_\SG$ the \define{simple network operad}. By Eq.\ \ref{eq:operations_in_CN}, an operation in $\CN_\SG(n_1, \dots, n_k; k)$ is an element of $S_n$ together with a simple graph having $\mathbf{n} = \{1, \dots, n\}$ as its set of vertices. 
\end{ex}

The operads coming from other one-colored network models work similarly. For example, if $\DG \maps \S \to \Mon$ is the network model of directed graphs from Ex.\ \ref{ex:directed_graph}, then an operation in $\CN_\SG(n_1, \dots, n_k; n)$ is an element of $S_n$ together with a directed graph having $\n$ as its set of vertices.

In Thm.\ \ref{thm:equations} we gave a pedestrian description of one-colored network models. We can describe the corresponding network operads in the same style:

\begin{thm}
\label{thm:one-colored_network_operads}
Suppose $F$ is a one-colored network model. Then the network operad $\CN_F$ is the $\N$-typed operad for which the following hold:
\begin{enumerate}
    \item The sets of operations are given by
    \[\CN_F(n_1, \dots, n_k; n) = \left\{ 
    \begin{array}{cl}  S_n \times F(n) & \textrm{if } 
    n_1 + \cdots n_k = n \\
    \emptyset & \textrm{otherwise.} 
    \end{array}  \right. \]
    \item Composition of operations is given as follows. Suppose that
    \[(\sigma,g) \in S_n \times F(n) = \CN_F(n_1, \dots, n_k; n) \]
    and for $1 \le i \le k$ we have
    \[(\tau_i,h_i) \in S_{n_i} \times F(n_i) =
    \CN_F(n_{i1}, \dots, n_{ij_i}; n_i). \]
    Then 
    \[(\sigma,g) \circ ((\tau_1,h_1), \dots, (\tau_k,h_k)) = (\sigma (\tau_1 + \cdots + \tau_k), g \cup \sigma(h_1 \sqcup \cdots \sqcup h_k)) \]
    where $+$ is defined in Eq.\ \ref{eq:plus}, while $\cup$ and $\sqcup$ are defined in Thm.\ \ref{thm:equations}.
    \item The identity operation in $\CN_F(n;n)$ is 
    $(1,e_n)$, where $1$ is the identity in $S_n$ and $e_n$ is the identity in the monoid $F(n)$.
    \item The right action of the symmetric group $S_k$ on $\CN_F(n_1, \dots, n_k;n)$ is given as follows. Given $(\sigma,g) \in \CN_F(n_1, \dots, n_k;n)$ and $\tau \in S_k$, we have
    \[(\sigma,g) \tau = (\sigma\tau, g) . \]
\end{enumerate}
\end{thm}

\begin{proof}
To prove these we apply Prop.\ \ref{prop:operad_from_symmoncat}, which describes the operad $\op(\C)$ coming from a strict symmetric monoidal category $\C$, to the case $\C = \Int F$. Item 1 is simply Eq.\ \ref{eq:operations_in_CN}. To prove item 2 we first use Eq.\ \ref{eq:composition_of_operations}, which defines composition of operations in $\op(\C)$ in terms of composition and tensoring of morphisms in $\C$. Then we use Eq.\ \ref{eq:composition}, which says how to compose morphisms in $\Int F$, and Eq.\ \ref{eq:tensoring_morphisms_in_G_construction}, which says how to tensor them. Item 3 comes from how identity operations in $\op(C)$ and identity morphisms in $\Int F$ are defined. Similarly, item 4 comes from how the symmetric group actions in $\op(C)$ and the braiding in $\Int F$ are defined.
\end{proof}

The construction of operads from symmetric monoidal categories described in Prop.\ \ref{prop:operad_from_symmoncat} is functorial, so the construction of operads from network models is as well. To discuss this functoriality we need a couple of categories.  The first is $\SSMC$, defined in Def.\
\ref{defn:SSMC}.  Second:

\begin{defn}
\label{defn:category_of_operads} 
Let $\Op$ be the category with typed operads as objects and with a morphism from the $T$-typed operad $O$ to the $T'$-typed operad $O'$ being a function $F \maps T \to T'$ together with maps
\[F \maps O(t_1, \dots, t_n ; t) \to 
O'(F(t_1), \dots, F(t_n); F(t)) \]
preserving the composition of operations, identity operations and the symmetric group actions.
\end{defn}

\begin{prop}
\label{prop:functoriality_of_operads_from_ssmcs}
There exists a unique functor $\op \maps \SSMC \to \Op$ defined on objects as in Prop.\ \ref{prop:operad_from_symmoncat} 
and sending any strict symmetric monoidal functor $F \maps \C \to \C'$ to the  operad morphism $\op(F) \maps \op(\C) \to \op(\C')$ that acts by $F$ on types and also on operations:
\[\op(F) = F \maps 
\hom_C(c_1 \otimes \cdots \otimes c_n, c) \to \hom_{\C'}(F(c_1) \otimes \cdots \otimes F(c_n), F(c)). \] 
\end{prop}

\begin{proof} This is a straightforward verification.
\end{proof}

\begin{thm} 
\label{thm:O}
There exists a unique 
functor
\[\CN \maps \NM \to \Op \]
sending any network model $F \maps \S(C) \to \Cat$ to the operad $\CN_F = \op(\Int G)$ and any morphism of network models $(G,\gn) \maps (C,F) \to (C',F' G')$ to the morphism of operads $\CN_G = \op(\Ghat)$.
\end{thm}

\begin{proof}
    There is a functor 
    \[\textstyle{\Int} \maps \NM \to \SSMC \]
    given by restricting the strict symmetric monoidal Grothendieck construction of Thm.\ \ref{thm:G_construction_functorial_for_strict_symmetric_cats} to $\NM$. Composing this with the functor
    \[\op \maps \SSMC \to \Op  \]
    constructed in Prop.\ \ref{prop:functoriality_of_operads_from_ssmcs} we obtain a functor $\CN \maps \NM \to \Op$ with the properties stated in the theorem.  Since these properties specify how $\CN$ acts on objects and morphisms, it is unique.
\end{proof}

\section{Algebras of network operads}
\label{sec:algebras}

Our interest in network operads comes from their use in designing and tasking networks of mobile agents. The operations in a network operad are ways of assembling larger networks of a given kind from smaller ones. To describe how these operations act in a concrete situation we need to specify an algebra of the operad. The flexibility of this approach to system design takes advantage of the fact that a single operad can have many different algebras, related by homomorphisms. We have already discussed these ideas elsewhere \cite{CommNet, CompTask}, and plan to write a more detailed treatment, so here we simply describe a few interesting algebras of network operads. 

Recall from the introduction that an algebra $A$ of a typed operad $O$ specifies a set $A(t)$ for each type $t \in T$ such that the operations of $O$ can be applied to act on these sets. That is, each algebra $A$ specifies: 

\begin{itemize}
    \item for each type $t \in T$, a set $A(t)$, and
    \item for any types $t_1, \dots, t_n, t \in T$, a function
    \[\alpha \maps O(t_1, \dots, t_n;t) \to \hom(A(t_1) \times \cdots \times A(t_n), A(t)) \] 
\end{itemize}
obeying some rules that generalize those for the action of a monoid on a set \cite{Yau}. All the examples in this section are algebras of network operads constructed from one-colored network models $F \maps \S \to \Mon$. This allows us to use Thm.\ \ref{thm:one-colored_network_operads}, which describes $\CN_F$ explicitly.

The most basic algebra of such a network operad $\CN_F$ is its `canonical algebra', where it acts on the kind of network described by the network model $F$:

\begin{ex}[\textbf{The canonical algebra}]
\label{ex:canonical_algebra}
Let $F \maps \S \to \Mon$ be a one-colored network model. Then the operad $\CN_F$ has a \define{canonical algebra} $A_F$ with
\[A_F(n) = F(n) \]
for each $n \in N$, the type set of $\CN_F$.
In this algebra any operation
\[(\sigma,g) \in  \CN_F(n_1, \dots , n_k; n) = 
S_n \times F(n) \] 
acts on a $k$-tuple of elements
\[h_i \in A_F(n_i) = F(n_i)   \qquad (1 \le i \le k)\]
to give
\[\alpha(\sigma,g)(h_1, \dots, h_k) =  g \cup \sigma(h_1 \sqcup \cdots \sqcup h_k) \in A(n) .\]
Here we use Thm.\ \ref{thm:equations}, which gives us the ability to overlay networks using the monoid structure $\cup \maps F(n) \times F(n) \to F(n)$, take their `disjoint union' using maps $\sqcup \maps F(m) \times F(m') \to F(m + m')$, and act on $F(n)$ by elements of $S_n$. Using the equations listed in this theorem one can check that $\alpha$ obeys the axioms of an operad algebra.
\end{ex}

When we want to work with networks that have more properties than those captured by a given network model, we can equip elements of the canonical algebra with extra attributes. Three typical kinds of network attributes are vertex attributes, edge attributes, and `global network' attributes. For our present purposes, we focus on vertex attributes. Vertex attributes can capture internal properties (or states) of agents in a network such as their locations,
capabilities, performance characteristics, etc.

\begin{ex}[\textbf{Independent vertex attributes}]
\label{ex:vertex_attribute_algebra}
   For any one-colored network model $F\maps \S \to \Mon$ and any set $X$, we can form an algebra $A_X$ of the operad $\O_F$ that consists of networks whose vertices have attributes taking values in $X$. To do this, we define 
   \[A_X(n) = F(n) \times X^n  .\]
   In this algebra, any operation
   \[(\sigma,g) \in  \CN_F(n_1, \dots , n_k; n) = 
    S_n \times F(n) \] 
   acts on a $k$-tuple of elements
   \[(h_i,x_i) \in F(n_i) \times X^{n_i} \qquad (1 \le i \le k) \]
   to give
   \[\alpha_X(\sigma,g) = (g \cup \sigma(h_1 \sqcup \cdots \sqcup h_k), \sigma(x_1, \dots, x_k)). \]
   Here $(x_1, \dots, x_k) \in X^n$
   is defined using the canonical bijection
   \[X^n \cong \prod_{i=1}^k X^{n_i} \]
   when $n_1 + \cdots + n_k = n$, and $\sigma \in S_n$ acts on $X^n$ by permutation of coordinates. In other words,  $\alpha_X$ acts via $\alpha$ on the  $F(n_i)$ factors while permuting the vertex attributes $X^n$ in the same way that the vertices of the network $h_1 \sqcup \cdots \sqcup h_k$ are permuted.

   One can easily check that the projections 
   $F(n) \times X^n \to F(n)$ 
   define a homomorphism of $\CN_F$-algebras, which we call
    \[\pi_X \maps A_X \to A  .\]
    This homomorphism `forgets the vertex attributes' taking values in the set $X$.
\end{ex}

\begin{ex}[\textbf{Simple networks with a rule obeyed by edges}]
\label{ex:edge_exception_algebra}
Let $\CN_\SG$ be the simple network operad as explained in Ex.\ \ref{ex:simple_network_operad}. We can form an algebra of the operad $\CN_\SG$ that consists of simple graphs whose vertices have attributes taking values in some set $X$, but where an edge is permitted between two vertices only if their attributes obey some condition. We specify this condition using a symmetric function 
\[p \maps X\times X \to \Boole \] 
where $\Boole = \{F,T\}$. An edge is not permitted between vertices with attributes $(x_1, x_2) \in X \times X$ if this function evaluates to $F$.

To define this algebra, which we call $A_p$, we let $A_p(n) \subseteq \SG(n) \times X^n$ be the set of pairs $(g,x)$ such that for all edges $\{i,j\} \in g$ the attributes of the vertices $i$ and $j$ make $p$ true:
\[p(x(i), x(j)) = T .\]
There is a function 
\[\tau_p \maps A_X(n) \to A_p(n) \]
that discards edges $\{i,j\}$ for which $p(x(i),x(j)) = F$.  Recall that $A_X(n) = \SG(n) \times X^n$, and recall from Ex.\ \ref{ex:simple_graph_2} that we can regard $\SG(n)$ as the set of functions $g \maps \E(n) \to \Boole$.   Then we define $\tau_p$ by
\[ \tau_p(g,x) = (\overline{g},x) \]
where
\[
    \overline{g}\{i,j\} = 
    \left\{  \begin{array}{ccl} g\{i,j\} & \textrm{if} & p(x(i),x(j)) = T \\ \\
    F & \textrm{if} & p(x(i), x(j)) = F.
    \end{array} \right.
\]

We can define an action $\alpha_p$ of $\CN_\SG$ on the sets $A_p(n)$ with the help of this function. Namely, we take $\alpha_p$ to be the composite
\[\begin{tikzcd}
\CN_\SG(n_1, \dots, n_k ; n) \times A_p(n_1) \times \cdots \times A_p(n_k)  \arrow[d, hookrightarrow]
\\
\CN_\SG(n_1, \dots, n_k ; n) \times A_X(n_1) \times \cdots \times A_X(n_k)  \arrow[d, "\alpha_X"]
\\
A_X(n)\arrow[d,  "\tau_p"]
\\
A_p(n)
\end{tikzcd}\]
where the action $\alpha_X$ was defined in Ex.\ \ref{ex:vertex_attribute_algebra}. One can check that $\alpha_p$ makes the sets $A_p(n)$ into an algebra of $\CN_\SG$, which we call $A_p$. One can further check that the maps $\tau$ define a homomorphism of $\CN_\SG$-algebras, which we call
\[\tau_p \maps A_X \to A_p  .\]
\end{ex}

\begin{ex}[\textbf{Range-limited networks}]
\label{ex:range_limit_algebra}
We can use the previous examples to model range-limited communications between entities in a plane. First, let $X = \R^2$ and form the algebra $A_X$ of the simple network operad $\CN_\SG$. Elements of $A_X(n)$ are simple graphs with vertices in the plane. 

Then, choose a real number $L \ge 0$ and let $d$ be the usual Euclidean distance function on the plane. Define $p \maps X \times X \to \Boole$ by setting $p(x, y)=T$ if $d(x,y) \le L$ and $p(x,y) = F$ otherwise. Elements of $A_p(n)$ are simple graphs with vertices in the plane such that no edge has length greater than $L$.
\end{ex}
   
\begin{ex}[\textbf{Networks with edge count limits}]
\label{ex:edge_count_algebra}
    Recall the network model for multigraphs $\MGplus$, defined in Ex.\ \ref{ex:multigraph} and clarified in Ex.\ \ref{ex:multigraph_2}. An element of $\MGplus(n)$ is a multigraph on the set $\n$, namely a function $g \maps \E(n) \to \N$ where $\E(n)$ is the set of 2-element subsets of $\n$.
    If we fix a set $X$ we obtain an algebra $A_X$ of $\CN_{\MGplus}$ as in Ex.\ \ref{ex:vertex_attribute_algebra}. The set $A_X(n)$ consists of multigraphs on $\n$ where the vertices have attributes taking values in $X$. 
    
    Starting from $A_X$ we can form another algebra  where there is an upper bound on how many edges are allowed between two vertices, depending on their attributes. We specify this upper bound using a symmetric function
    \[b \maps X \times X \to \N. \]
    
    To define this algebra, which we call $A_b$, let
    $A_b(n) \subseteq \MGplus(n) \times X^n$ be the set of pairs $(g,x)$ such that for all $\{i,j\} \in \E(n)$ we have 
    \[g(i,j) \le b(x(i), x(j)) .\]
    Much as in Ex.\ \ref{ex:edge_exception_algebra} there is function 
    \[\pi \maps A_X(n) \to A_b(n) \]
    that enforces this upper bound: for each $g \in A_X(n)$ its image  $\pi(g)$ is obtained by reducing the number of edges between vertices $i$ and $j$ to the minimum of $g(i,j)$ and $\beta(i,j)$:
    \[\pi(g)(i,j) = g(i,j) \min \beta(i,j) .\]
    We can define an action $\alpha_b$ of $\CN_\MG$ on the sets $A_b(n)$ as follows:
    \[\begin{tikzcd}
    \CN_\MG(n_1, \dots, n_k ; n) \times A_p(n_1) \times \cdots \times A_p(n_k)  \arrow[d, hookrightarrow]
    \\
    \CN_\MG(n_1, \dots, n_k ; n) \times A_X(n_1) \times \cdots \times A_X(n_k)  \arrow[d, "\alpha_X"]
    \\
    A_X(n)\arrow[d,  "\pi"]
    \\
    A_p(n)
    \end{tikzcd}\]
    One can check that $\alpha_b$ indeed makes the sets $A_b(n)$ into an algebra of $\CN_\MGplus$, which we call $A_b$, and that the maps $\pi_p$ define a homomorphism of $\CN_\MGplus$-algebras, which we call
    \[\pi_p \maps A_X \to A_b  .\]
\end{ex}
    
\begin{ex}[\textbf{Range-limited networks, revisited}]
\label{ex:range_limit_algebra_2}
We can use Ex.\ \ref{ex:edge_count_algebra} to model entities in the plane that have two types of communication channel, one of which has range $L_1$ and another of which has a lesser range $L_2 < L_1$. To do this, take $X = \R^2$ and define $b \maps X \times X \to \N$ by 
\[b(x, y)= \left\{ \begin{array}{cl}
0 & \textrm{if } d(x,y) > L_1 \\
1 & \textrm{if } L_2 < d(x,y) \le L_1  \\
2 & \textrm{if } d(x,y) \le L_2 
\end{array}  \right.
\]
Elements of $A_b(n)$ are multigraphs with vertices in the plane having no edges between vertices whose distance is $> L_1$, at most one edge between vertices whose distance is $\le L_1$ but $> L_2$, and at most two edges between vertices whose distance is $\le L_2$. 

Moreover, the attentive reader may notice that the action $\alpha_b$ of $\CN_\MGplus$ for this specific choice of $b$ factors through an action of $\CN_{\G_{\Boole_2}}$, where $\G_{\Boole_2}$ is the network model defined in Ex.\ \ref{ex:multigraph_with_at_most_k}. That is, operations $\CN_{\G_{\Boole_2}}(n_1, \dots , n_k; n) = S_n \times \G_{\Boole_2}(n)$ where
$\G_{\Boole_2}(n)$ is the set of multigraphs on $\n$ \emph{with at most $2$ edges between vertices} are sufficient to compose these range-limited networks. This is due to the fact that the values of this $b \maps X \times X \to \N$ are at most 2. 
\end{ex}

These examples indicate that vertex attributes and constraints can be systematically added to the canonical algebra to build more interesting algebras, which are related by homomorphisms. Ex.\ \ref{ex:vertex_attribute_algebra} illustrates how adding extra attributes to the networks in some algebra $A$ can give networks that are elements of an algebra $A'$ equipped with a homomorphism $\pi \maps A' \to A$ that forgets these extra attributes. Ex.\ \ref{ex:edge_count_algebra} illustrates how imposing extra constraints on the networks in some algebra $A$ can give an algebra $A'$ equipped with a homomorphism $\tau \maps A \to A'$ that imposes these constraints: this works only if there is a well-behaved systematic procedure, defined by $\tau$, for imposing the constraints on any element of $A$ to get an element of $A'$.

The examples given so far scarcely begin to illustrate the rich possibilities of network operads and their algebras.  Their connection to Petri nets is explored in \cite{Catalysts}, but there is much more to do.

In particular, it is worth noting that all the specific examples of network models described here involve commutative monoids. However, noncommutative monoids are also important. Suppose, for example, that we wish to model entities with a limited number of point-to-point communication interfaces---e.g., devices with a finite number $p$ of USB ports. More formally, we wish to act on sets of degree-limited networks $A_{\rm deg} (n)\subset \SG(n) \times \N^n$ made up  of pairs $(g, p)$ such that the degree of each vertex $i$, ${\rm deg}(i),$ is at most the degree-limiting attribute of $i$: ${\rm deg}(i) \le p(i)$. Na\"ively, we might attempt to construct a map $\tau_{\rm deg} \maps A_\N \to A_{\rm deg}$ as in Ex.\ \ref{ex:edge_count_algebra} to obtain an action of the simple network operad $\CN_\SG$. However, this is turns out to be impossible. For example, if attempt to build a network from devices with a single USB port, and we attempt to connect multiple USB cables to one of these devices, the relevant network operad must include a rule saying which attempts, if any, are successful. Since we cannot prioritize links from some vertices over others---which would break the symmetry built into any network model---the order in which these attempts are made must be relevant. Since the monoids $\SG(n)$ are commutative, they cannot capture this feature of the situation.

The solution is to use a class of noncommutative monoids dubbed `graphic monoids' by Lawvere \cite{Law}: namely, those that obey the identity $aba =ab$. These allow us to construct a one-colored network model $\Gamma \maps \S \to \Mon$ whose network operad $\CN_\Gamma$ acts on $A_{\rm deg}$. For our USB device example, the relation $aba = ab$ means that first attempting to connect some USB cables between some devices ($a$), second attempting to connect some further USB cables ($b$), and third attempting to connect some USB cables \emph{precisely as attempted in the first step} ($a$, again) has the same result as only performing the first two steps ($ab$).

In fact, one-colored network models constructed from graphic monoids appear to be sufficient to model a wide array of constraints on the structural design and behavioral tasking of agents. For more on network models arising from noncommutative monoids, see \cite{Moeller}. 

\subsection*{Acknowledgements}

This work was supported by the DARPA Complex Adaptive System Composition and Design Environment (CASCADE) project. We thank Chris Boner, Tony Falcone, Marisa Hughes, Joel Kurucar, Tom Mifflin, John Paschkewitz, Thy Tran and Didier Vergamini for many helpful discussions. Christina Vasilakopoulou provided crucial assistance with Sec.\ \ref{subsec:mon_fib}. JB also thanks the Centre for Quantum Technologies, where some of this work was done.

\end{document}